\documentclass[11pt,a4paper]{article}
\usepackage[utf8]{inputenc}
\usepackage[english]{babel}
\usepackage[T1]{fontenc}
\usepackage{amsmath}
\usepackage{amsfonts}
\usepackage{amssymb}
\usepackage{amsthm}
\usepackage{mathrsfs}
\usepackage{dsfont}
\usepackage[left=2cm,right=2cm,top=2.5cm,bottom=2.5cm]{geometry}
\usepackage{multirow}
\usepackage{hhline}

\RequirePackage[colorlinks,citecolor=blue,urlcolor=blue]{hyperref}

\newtheorem{theorem}{Theorem}

\newtheorem{proposition}{Proposition}
\newtheorem{problem}{Problem}
\newtheorem{definition}{Definition}
\newtheorem{algorithm}{Algorithm}

\DeclareMathOperator{\argmin}{argmin}

\title{Convergence rates of Gibbs measures with degenerate minimum}
\author{Pierre Bras\footnote{Sorbonne Universit\'e, Laboratoire de Probabilit\'es, Statistique et Mod\'elisation, UMR 8001, case 188, 4 pl. Jussieu, F-75252 Paris Cedex 5, France. E-mail: \texttt{pierre.bras@sorbonne-universite.fr}}}
\date{}

\begin{document}

\maketitle

\begin{abstract}
We study convergence rates of Gibbs measures, with density proportional to $e^{-f(x)/t}$, as $t \rightarrow 0$ where $f : \mathbb{R}^d \rightarrow \mathbb{R}$ admits a unique global minimum at $x^\star$.
We focus on the case where the Hessian is not definite at $x^\star$. We assume instead that the minimum is strictly polynomial and give a higher order nested expansion of $f$ at $x^\star$, which depends on every coordinate. We give an algorithm yielding such an expansion if the polynomial order of $x^\star$ is no more than $8$, in connection with Hilbert's $17^{\text{th}}$ problem. However, we prove that the case where the order is $10$ or higher is fundamentally different and that further assumptions are needed. We then give the rate of convergence of Gibbs measures using this expansion. Finally we adapt our results to the multiple well case.
\end{abstract}

\section{Introduction}

Gibbs measures and their convergence properties are often used in stochastic optimization to minimize a function defined on $\mathbb{R}^d$. That is, let $f : \mathbb{R}^d \rightarrow \mathbb{R}$ be a measurable function and let $x^\star \in \mathbb{R}^d$ be such that $f$ admits a global minimum at $x^\star$. It is well known \cite{hwang1980} that under standard assumptions, the associated Gibbs measure with density proportional to $e^{-f(x)/t}$ for $t >0$, converges weakly to the Dirac mass at $x^\star$, $\delta_{x^\star}$, when $t \rightarrow 0$. The Langevin equation $dX_s = -\nabla f(X_s) ds + \sigma dW_s$ consists in a gradient descent with Gaussian noise. For $\sigma = \sqrt{2t}$, its invariant measure has a density proportional to $e^{-f(x)/t}$ (see for example \cite{khasminskii2012}, Lemma 4.16), so for small $t$ we can expect it to converge to $\text{argmin}(f)$ \cite{dalalyan2016} \cite{barrera2020}.
The simulated annealing algorithm \cite{laarhoven1987} builds a Markov chain from the Gibbs measure where the parameter $t$ converges to zero over the iterations.
This idea is also used in \cite{gelfand-mitter}, giving a stochastic gradient descent algorithm where the noise is gradually decreased to zero.
Adding a small noise to the gradient descent allows to explore the space and to escape from traps such as local minima and saddle points which appear in non-convex optimization problems \cite{lazarev1992} \cite{dauphin2014}.
Such methods have been recently brought up to light again with SGLD (Stochastic Gradient Langevin Dynamics) algorithms \cite{welling2011} \cite{li2015}, especially for Machine Learning and calibration of artificial neural networks, which is a high-dimensional non-convex optimization problem.

\medskip

The rates of convergence of Gibbs measures have been studied in \cite{hwang1980}, \cite{hwang1981} and \cite{athreya2010} under differentiability assumptions on $f$. It turns out to be of order $t^{1/2}$ as soon as the Hessian matrix $\nabla^2 f(x^\star)$ is positive definite. Furthermore, in the multiple well case i.e. if the minimum of $f$ is attained at finitely many points $x_1^\star$, $\ldots$, $x_m^\star$, \cite{hwang1980} proves that the limit distribution is a sum of Dirac masses $\delta_{x_i^\star}$ with coefficients proportional to $\det(\nabla^2 f(x_i^\star))^{-1/2}$ as soon as all the Hessian matrices are positive definite. If such is not the case, we can conjecture that the limit distribution is concentrated around the $x_i^\star$ where the degeneracy is of the highest order.

\medskip

The aim of this paper is to provide a rate of convergence in this degenerate setting, i.e. when $x^\star$ is still a strict global minimum but $\nabla^2 f(x^\star)$ is no longer definite, which extends the range of applications of Gibbs measure-based algorithms where positive definiteness is generally assumed.
A general framework is given in \cite{athreya2010}, which provides rates of convergence based on dominated convergence. However a strong and rather technical assumption on $f$ is needed and checking it seems, to some extent, more demanding than proving the result. To be more precise, the assumption reads as follows: there exists a function $g : \mathbb{R}^d \rightarrow \mathbb{R}$ with $e^{-g} \in L^1(\mathbb{R}^d)$ and $\alpha_1, \ \ldots, \ \alpha_d \in (0,+\infty)$ such that
\begin{equation}
\label{eq:intro}
\forall h \in \mathbb{R}^d, \ \ \frac{1}{t} \left[ f(x^\star + \left( t^{\alpha_1} h_1,\ldots, t^{\alpha_d} h_d \right) ) - f(x^\star) \right] \underset{t \rightarrow 0}{\longrightarrow} g(h_1,\ldots,h_d).
\end{equation}

Our objective is to give conditions on $f$ such that \eqref{eq:intro} is fulfilled and then to elucidate the expression of $g$ depending on $f$ and its derivatives by studying the behaviour of $f$ at $x^\star$ in every direction. Doing so we can apply the results from \cite{athreya2010} yielding the convergence rate of the corresponding Gibbs measures. The orders $\alpha_1$, $\ldots$, $\alpha_d$ must be chosen carefully and not too big, as the function $g$ needs to depend on every of its variables $h_1$, $\ldots$, $h_d$, which is a necessary condition for $e^{-g}$ to be integrable.
We also extend our results to the multiple well case.

We generally assume $f$ to be coercive, i.e. $f(x) \rightarrow + \infty$ as $||x|| \rightarrow + \infty$, $\mathcal{C}^{2p}$ in a neighbourhood of $x^\star$ for some $p \in \mathbb{N}$ and we assume that the minimum is polynomial strict, i.e. the function $f$ is bounded below in a neighbourhood of $x^\star$ by some non-negative polynomial function, null only at $x^\star$. Thus we can apply a multi-dimensional Taylor expansion to $f$ at $x^\star$, where the successive derivatives of $f :\mathbb{R}^d \rightarrow \mathbb{R}$ are seen as symmetric tensors of $\mathbb{R}^d$. The idea is then to consider the successive subspaces where the derivatives of $f$ are null up to some order ; using that the Taylor expansion of $f(x^\star+h)-f(x^\star)$ is non-negative, some cross derivative terms are null. However a difficulty arises at orders $6$ and higher, as the set where the derivatives of $f$ are null up to some order is no longer a vector subspace in general. This difficulty is linked with Hilbert's $17^{\text{th}}$ problem \cite{hilbert1888}, stating that a non-negative multivariate polynomial cannot be written as the sum of squares of polynomials in general. We thus need to change the definition of the subspaces we consider.
Following this, we give a recursive algorithm yielding an adapted decomposition of $\mathbb{R}^d$ into vector subspaces and a function $g$ satisfying \eqref{eq:intro} up to a change of basis, giving a canonical higher order nested decomposition of $f$ at $x^\star$ in degenerate cases. An interesting fact is that the case where the polynomial order of $x^\star$ is $10$ or higher fundamentally differs from those of orders $2$, $4$, $6$ and $8$, owing to the presence of even cross terms which may be not null. The algorithm we provide works at the orders $10$ or higher only under the assumption that all such even cross terms are null. In general, it is more difficult to get a general expression of $g$ for the orders $10$ and higher.
We then apply our results to \cite{athreya2010}, where we give conditions such that the hypotheses of \cite{athreya2010}, especially \eqref{eq:intro}, are satisfied so as to infer rates of convergence of Gibbs measures in the degenerate case where $\nabla^2 f(x^\star)$ is not necessarily positive definite.
The function $g$ given by our algorithm is a non-negative polynomial function and non-constant in any of its variables, however it needs to be assumed to be coercive to be applied to \cite{athreya2010}. We study the case where $g$ is not coercive and give a method to deal with simple generic non-coercive cases, where our algorithm seems to be a first step to a more general procedure. However, we do not give a general method in this case.

Our results are applied to Gibbs measures but they can also be applied to more general contexts, as we give a canonical higher order nested expansion of $f$ at a minimum, in the case where some derivatives are degenerate.

For general properties of symmetric tensors we refer to \cite{comon2008}. In the framework of stochastic approximation, \cite{fort1999} Section 3.1 introduced the notion of strict polynomial local extremum and investigated their properties as higher order "noisy traps".

\medskip

The paper is organized as follows. In Section \ref{section:gibbs_measures}, we recall convergence properties of Gibbs measures and revisit the main theorem from \cite{athreya2010}. This theorem requires, as an hypothesis, to find an expansion of $f$ at its global minimum ; we properly state this problem in Section \ref{subsection:statement_of_problem} under the assumption of strict polynomial minimum. In Section \ref{section:main_result}, we state our main result for both single well and multiple well cases, as well as our algorithm.
In Section \ref{section:expansion}, we detail the expansion of $f$ at its minimum for each order and provide the proof. We give the general expression of the canonical higher order nested expansion at any order in Section \ref{sec:expansion_any_order}, where we distinguish the orders $10$ and higher from the lower ones. We then provide the proof for each order $2$, $4$, $6$ and $8$ in Sections \ref{section:order_2}, \ref{section:order_4}, \ref{section:order_6} and \ref{section:order_8} respectively. We need to prove that, with the exponents $\alpha_1$, $\ldots$, $\alpha_d$ we specify, the convergence in \eqref{eq:intro} holds ; we do so by proving that, using the non-negativity of the Taylor expansion, some cross derivative terms are zero. Because of Hilbert's $17^{\text{th}}$ problem, we need to distinguish the orders $6$ and $8$ from the orders $2$ and $4$, as emphasized in Section \ref{section:hilbert}. For orders $10$ and higher, such terms are not necessarily zero and must then be assumed to be zero. In Section \ref{section:order_10}, we give a counter-example if this assumption is not satisfied before proving the result.
In Section \ref{subsec:unif_non_constant}, we prove that for every order the resulting function $g$ is constant in none of its variables and that the convergence in \eqref{eq:intro} is uniform on every compact set.
In Section \ref{section:non_coercive}, we study the case where the function $g$ is not coercive and give a method to deal with the simple generic case.
In Section \ref{section:proofs_athreya}, we prove our main theorems stated in Section \ref{section:main_result} using the expansion of $f$ established in Section \ref{section:expansion}. Finally, in Section \ref{section:flat}, we deal with a "flat" example where all the derivatives in the local minimum are zero and where we cannot apply our main theorems.

\section{Definitions and notations}

We give a brief list of notations that are used throughout the paper.

We endow $\mathbb{R}^d$ with its canonical basis $(e_1,\ldots,e_d)$ and the Euclidean norm denoted by $|| \boldsymbol{\cdot} ||$. For $x \in \mathbb{R}^d$ and $r >0$ we denote by $\mathcal{B}(x,r)$ the Euclidean ball of $\mathbb{R}^d$ of center $x$ and radius $r$.
For $E$ a vector subspace of $\mathbb{R}^d$, we denote by $p_{_E} : \mathbb{R}^d \rightarrow E$ the orthogonal projection on $E$. For a decomposition of $\mathbb{R}^d$ into orthogonal subspaces, $\mathbb{R}^d = E_1 \oplus \cdots \oplus E_p$, we say that an orthogonal transformation $B \in \mathcal{O}_d(\mathbb{R})$ is adapted to this decomposition if for all $j \in \lbrace 1, \ldots, p \rbrace $,
$$ \forall	i \in \lbrace \dim(E_1)+\cdots+\dim(E_{j-1})+1, \ldots, \dim(E_1)+\cdots+\dim(E_j) \rbrace, \ B \cdot e_i \in E_j .$$
For $a, \ b \in \mathbb{R}^d$, we denote by $a \ast b$ the element-wise product, i.e.
$$ \forall i \in \lbrace 1, \ldots, d \rbrace, \ (a \ast b)_i = a_i b_i .$$
For $v^1$, $\ldots$, $v^k$ vectors in $\mathbb{R}^d$ and $T$ a tensor of order $k$ of $\mathbb{R}^d$, we denote the tensor product
$$ T \cdot (v^1 \otimes \cdots \otimes v^k) = \sum_{i_1,\ldots,i_k \in \{1,\ldots,d \}} T_{i_1 \cdots i_k} v^1_{i_1} \ldots v^k_{i_k} . $$
More generally, if $j \le k$ and $v^1, \ \ldots, \ v^j$ are $j$ vectors in $\mathbb{R}^d$, then  $T \cdot (v^1 \otimes \cdots \otimes v^j)$ is a tensor of order $k-j$ such that:
$$ T \cdot (v^1 \otimes \cdots \otimes v^j)_{i_{j+1}\ldots i_k} = \sum_{i_1,\ldots,i_j \in \{1,\ldots ,d \}} T_{i_1 \ldots i_k} v^1_{i_1} \ldots v^j_{i_j}. $$
For $h \in \mathbb{R}^d$, $h^{\otimes k}$ denotes the tensor of order $k$ such that
$$ h^{\otimes k} = (h_{i_1} \ldots h_{i_k} )_{i_1,\ldots,i_k \in \{1,\ldots,d \}}.$$

For a function $f \in \mathcal{C}^p\left(\mathbb{R}^d, \mathbb{R}\right)$, we denote $\nabla^k f(x)$ the differential of order $k \le p$ of $f$ at $x$, as $\nabla^k f(x)$ is the tensor of order $k$ defined by:
$$ \nabla^k f(x) = \left(\frac{\partial^k f(x)}{\partial x_{i_1} \cdots \partial x_{i_k}}\right)_{i_1,i_2,\ldots,i_k \in \{1,\ldots,d \}} .$$
By Schwarz's theorem, this tensor is symmetric, i.e. for all permutation $\sigma \in \mathfrak{S}_k$,
$$ \frac{\partial^k f(x)}{\partial x_{i_{\sigma(1)}} \cdots \partial x_{i_{\sigma(k)}}} = \frac{\partial^k f(x)}{\partial x_{i_1} \cdots \partial x_{i_k}} .$$
We recall the Taylor-Young formula in any dimension, and the Newton multinomial formula.
\begin{theorem}[Taylor-Young formula]
\label{theorem:taylor}
Let $f : \mathbb{R}^d \rightarrow \mathbb{R}$ be $\mathcal{C}^p$ and let $x \in \mathbb{R}^d$. Then:
$$ f(x + h) \underset{h \rightarrow 0}{=} \sum_{k=0}^{p} \frac{1}{k!} \nabla^k f(x) \cdot h^{\otimes k} + ||h||^p o(1) .$$
\end{theorem}
\noindent We denote by $\binom{k}{i_1, \ldots, i_p}$ the $p$-nomial coefficient, defined as:
$$ \binom{k}{i_1,\ldots ,i_p} = \frac{k!}{i_1!\ldots i_p!} .$$
\begin{theorem}[Newton multinomial formula]
Let $h_1, \ \ldots, \ h_p \in \mathbb{R}^d$, then
\begin{equation}
\label{equation:multinomial}
(h_1 + h_2 + \cdots + h_p)^{\otimes k} = \sum_{\substack{i_1,\ldots,i_p \in \lbrace 0,\ldots,k \rbrace \\ i_1 + \cdots + i_p = k }} \binom{k}{i_1, \ldots, i_p} h_1^{\otimes i_1} \otimes \cdots \otimes h_p^{\otimes i_p} .
\end{equation}
\end{theorem}
\noindent For $T$ a tensor of order $k$, we say that $T$ is non-negative (resp. positive) if
\begin{equation}
\label{eq:tensor_positive_def}
\forall h \in \mathbb{R}^d, \ T \cdot h^{\otimes k} \ge 0 \text{ (resp. } T \cdot h^{\otimes k} > 0 \text{)}.
\end{equation}
We denote $L^1(\mathbb{R}^d)$ the set of measurable functions $f:\mathbb{R}^d \rightarrow \mathbb{R}$ that are integrable with respect to the Lebesgue measure on $\mathbb{R}^d$. We denote by $\lambda_d$ the Lebesgue measure on $\mathbb{R}^d$. For $f : \mathbb{R}^d \rightarrow \mathbb{R}$ such that $e^{-f} \in L^1(\mathbb{R}^d)$, we define for $t > 0$, $C_t := \left( \int_{\mathbb{R}^d} e^{-f/t} \right)^{-1}$ and $\pi_t$ the Gibbs measure
$$ \pi_t(x)dx := C_t e^{-f(x)/t} dx .$$
For a family of random variables $(Y_t)_{t \in (0,1]}$ and $Y$ a random variable, we write $Y_t \underset{t \rightarrow 0}{\overset{\mathscr{L}}{\longrightarrow}} Y$ meaning that $(Y_t)$ weakly converges to $Y$.

We give the following definition of a strict polynomial local minimum of $f$:
\begin{definition}
Let $f : \mathbb{R}^d \rightarrow \mathbb{R}$ be $\mathcal{C}^{2p}$ for $p \in \mathbb{N}$ and let $x^\star$ be a local minimum of $f$. We say that $f$ has a strict polynomial local minimum at $x^\star$ of order $2p$ if $p$ is the smallest integer such that:
\begin{equation}
\label{eq:polynomial_strict}
\exists r >0, \ \forall h \in \mathcal{B}(x^\star, r) \setminus \{0 \} , \ \sum_{k=2}^{2p} \frac{1}{k!} \nabla^k f (x^\star) \cdot h^{\otimes k} > 0 .
\end{equation}
\end{definition}
\textbf{Remarks :}
\begin{enumerate}
	\item A local minimum $x^\star$ of $f$ is not necessarily strictly polynomial, for example, $f : x \mapsto e^{-||x||^{-2}}$ and $x^\star = 0$.
	\item If $x^\star$ is polynomial strict, then the order is necessarily even, because if $x^\star$ is not polynomial strict of order $2l$ for some $l \in \mathbb{N}$, then we have $h_n \rightarrow 0$ such that the Taylor expansion in $h_n$ up to order $2l$ is zero ; by the minimum condition, the Taylor expansion in $h_n$ up to order $2l+1$ must be non-negative, so we also have $\nabla^{2l+1}f(x^\star) \cdot h_n^{\otimes 2l+1} = 0$.
\end{enumerate}

\medskip

For $f :\mathbb{R}^d \rightarrow \mathbb{R}$ such that $\min_{\mathbb{R}^d}(f)$ exists, we denote by $\text{argmin}(f)$ the arguments of the minima of $f$, i.e.
$$ \text{argmin}(f) = \left\lbrace x \in \mathbb{R}^d : \ f(x) = \min_{\mathbb{R}^d}(f) \right\rbrace .$$
Without ambiguity, we write "minimum" or "local minimum" to designate $f(x^\star)$ as well as $x^\star$.
Finally, we define, for $x^\star \in \mathbb{R}^d$ and $p \in \mathbb{N}$:
\begin{align*}
\mathscr{A}_p(x^\star) & := \left\lbrace f \in \mathcal{C}^{2p}(\mathbb{R}^d, \mathbb{R}) : \ f \text{ admits a local minimum at } x^\star \right\rbrace . \\
\mathscr{A}_p^\star(x^\star) & := \left\lbrace f \in \mathcal{C}^{2p}(\mathbb{R}^d, \mathbb{R}) : \ f \text{ admits a strict polynomial local minimum at } x^\star \text{ of order } 2p \right\rbrace .
\end{align*}

\section{Convergence of Gibbs measures}
\label{section:gibbs_measures}

\subsection{Properties of Gibbs measures}
Let us consider a Borel function $f : \mathbb{R}^d \rightarrow \mathbb{R}$ with $e^{-f} \in L^1(\mathbb{R}^d)$. We study the asymptotic behaviour of the probability measures of density for $t \in (0,\infty)$:
$$ \pi_t(x) dx = C_t e^{-\frac{f(x)}{t}} dx $$
when $t \rightarrow 0$. When $t$ is small, the measure $\pi_t$ tends to the set $\argmin(f)$. The following proposition makes this statement precise.

\begin{proposition}
\label{proposition:gibbs}
Let $f : \mathbb{R}^d \rightarrow \mathbb{R}$ be a Borel function such that
$$ f^\star := \textup{essinf}(f) = \inf \{y : \ \lambda_d\{f \le y \} >0 \} > - \infty ,$$
and $e^{-f} \in L^1(\mathbb{R}^d)$. Then
$$ \forall \varepsilon >0, \ \pi_t(\{ f \ge f^\star + \varepsilon\} ) \underset{t \rightarrow 0}{\longrightarrow} 0 .$$
\end{proposition}

\begin{proof}
As $f^\star > -\infty$, we may assume without loss of generality that $f^\star = 0$ by replacing $f$ by $f-f^\star$. Let $\varepsilon>0$. It follows from the assumptions that $f\ge 0$ $\lambda_d$-$a.e.$ and $\lambda_d \lbrace f \le \varepsilon \rbrace >0$ for every $\varepsilon>0$. As $e^-f  \in L^1(\mathbb{R}^d)$, we have
$$ \lambda_d \lbrace f \le \varepsilon/3 \rbrace \le e^{\varepsilon/3} \int_{\mathbb{R}^d} e^{-f} d\lambda_d<+\infty. $$
Moreover by dominated convergence, it is clear that 
$$ C_t^{-1} \downarrow  \lambda_d\lbrace f=0 \rbrace<+\infty. $$

We have
$$C_t \le \left(\int_{f \le \varepsilon/3} e^{-\frac{f(x)}{t}}dx \right)^{-1} \le \left( e^{-\frac{\varepsilon}{3t}} \underbrace{\lambda_d \{ f \le \frac{\varepsilon}{3} \} }_{>0} \right)^{-1}. $$
Then
\begin{align*}
\pi_t\{f \ge \varepsilon\} = C_t \int_{f \ge \varepsilon} e^{-\frac{f(x)}{t}}dx \le \frac{e^{\varepsilon/3t} \int_{f \ge \varepsilon} e^{-f(x) / t}dx }{\lambda_d \{ f \le \frac{\varepsilon}{3} \}} \le \frac{e^{-\varepsilon/3t}C_{3t}^{-1}}{\lambda_d \{ f \le \frac{\varepsilon}{3} \}} \underset{t \rightarrow 0}{\longrightarrow} 0,
\end{align*}
because if $f(x) \ge \varepsilon$, then $e^{-\frac{f(x)}{t}} \le e^{-\frac{2\varepsilon}{3t}}e^{-\frac{f(x)}{3t} }$, and where we used that $C_{3t}^{-1}\le C_1^{-1}$ if $t\le 1/3$
\end{proof}

Now, let us assume that $f : \mathbb{R}^d \rightarrow \mathbb{R}$ is continuous, $e^{-f} \in L^1(\mathbb{R}^d)$ and $f$ admits a unique global minimum at $x^\star$ so that $\text{argmin}(f) = \{ x^\star \}$.
In \cite{athreya2010} is proved the weak convergence of $\pi_t$ to $\delta_{x^\star}$ and a rate of convergence depending on the behaviour of $f(x^\star + h)-f(x^\star)$ for small enough $h$. Let us recall this result in detail ; we may assume without loss of generality that $x^\star=0$ and $f(x^\star) = 0$.

\begin{theorem}[Athreya-Hwang, 2010]
\label{theorem:athreya:1}
Let $f : \mathbb{R}^d \rightarrow [0,\infty)$ be a Borel function such that :
\begin{enumerate}
\item $e^{-f} \in L^1(\mathbb{R}^d)$.
\item For all $\delta > 0$, $\inf \{ f(x), \ ||x|| > \delta \} > 0$.
\item There exist $\alpha_1, \ \ldots, \ \alpha_d > 0$ such that for all $(h_1,\ldots,h_d) \in \mathbb{R}^d$,
$$ \frac{1}{t} f(t^{\alpha_1} h_1,\ldots, t^{\alpha_d} h_d) \underset{t \rightarrow 0}{\longrightarrow} g(h_1,\ldots,h_d) \in \mathbb{R}. $$
\item $ \displaystyle\int_{\mathbb{R}^d} \sup_{0<t<1} e^{-\frac{f\left(t^{\alpha_1}h_1,\ldots ,t^{\alpha_d}h_d\right)}{t}} dh_1\ldots dh_d < \infty$.
\end{enumerate}
For $0<t<1$, let $X_t$ be a random vector with distribution $\pi_t$. Then $e^{-g} \in L^1(\mathbb{R}^d)$ and
\begin{equation}
\label{equation:athreya_assumption:3}
\left( \frac{(X_t)_1}{t^{\alpha_1}}, \ldots, \frac{(X_t)_d}{t^{\alpha_d}} \right) \overset{\mathscr{L}}{\longrightarrow} X  \ \text{ as } t \rightarrow 0
\end{equation}
where the distribution of $X$ has a density proportional to $e^{-g(x_1,\ldots,x_d)}$.
\end{theorem}

\noindent \textbf{Remark:} Hypothesis 2. is verified as soon as $f$ is continuous, coercive (i.e. $f(x) \longrightarrow + \infty$ when $||x|| \rightarrow + \infty$) and that $\text{argmin}(f) = \lbrace 0 \rbrace$.

\medskip

To study the rate of convergence of the measure $\pi_t$ when $t \rightarrow 0$ using Theorem \ref{theorem:athreya:1}, we need to identify $\alpha_1,\ldots, \ \alpha_d$ and $g$ such that the condition \eqref{equation:athreya_assumption:3} holds, up to a possible change of basis.
Since $x^\star$ is a local minimum, the Hessian $\nabla^2 f(x^\star)$ is positive semi-definite. Moreover, if $\nabla^2 f(x^\star)$ is positive definite, then choosing $\alpha_1=\cdots=\alpha_d=\frac{1}{2}$, we have:
$$ \frac{1}{t} f(t^{1/2} h) \underset{t \rightarrow 0}{\longrightarrow} \frac{1}{2} h^T \cdot \nabla^2f(x^\star) \cdot h :=g(x) .$$
And using an orthogonal change of variable:
$$ \int_{\mathbb{R}^d} e^{-g(x)}dx = \int_{\mathbb{R}^d} e^{-\frac{1}{2} \sum_{i=1}^d \beta_i y_i^2} dy_1\ldots dy_d < \infty ,$$
where the eigenvalues $\beta_i$ are positive. However, if $\nabla^2f(x^\star)$ is not positive definite, then some of the $\beta_i$ are zero and the integral does not converge.

\subsection{Statement of the problem}
\label{subsection:statement_of_problem}

We still consider the function $f : \mathbb{R}^d \rightarrow \mathbb{R}$ and assume that $f \in \mathscr{A}_p^\star(x^\star)$ for some $x^\star \in \mathbb{R}^d$ and some integer $p \ge 1$. Then our objective is to find $\alpha_1 \ge \cdots \ge \alpha_d \in (0,+\infty)$ and an orthogonal transformation $B \in \mathcal{O}_d(\mathbb{R})$ such that:
\begin{equation}
\label{eq:alpha_developpement}
\forall h \in \mathbb{R}^d, \ \ \frac{1}{t} \left[ f(x^\star + B \cdot (t^\alpha \ast h)) - f(x^\star) \right] \underset{t \rightarrow 0}{\longrightarrow} g(h_1,\ldots,h_d),
\end{equation}
where $t^\alpha$ denotes the vector $(t^{\alpha_1},\ldots,t^{\alpha_d})$ and where $g : \mathbb{R}^d \rightarrow \mathbb{R}$ is a measurable function which is not constant in any $h_1, \ \ldots, \ h_d$, i.e. for all $i \in \lbrace 1,\ldots,d \rbrace$, there exist $h_1, \ \ldots, \ h_{i-1}, \ h_{i+1},\ \ldots, \ h_d \in \mathbb{R}^d$ such that
\begin{equation}
\label{eq:def_non_constant}
h_i \mapsto g(h_1,\ldots,h_d) \text{ is not constant.}
\end{equation}
Then we say that $\alpha_1, \ \ldots, \ \alpha_d$, $B$ and $g$ are a solution of the problem \eqref{eq:alpha_developpement}.
The hypothesis that $g$ is not constant in any of its variables is important ; otherwise, we could simply take $\alpha_1 = \cdots = \alpha_d = 1$ and obtain, by the first order condition:
$$ \frac{1}{t} \left[ f(x^\star + t(h_1,\ldots,h_d)) - f(x^\star) \right] \underset{t \rightarrow 0}{\longrightarrow} 0 .$$

\subsection{Main results : rate of convergence of Gibbs measures}
\label{section:main_result}

\begin{theorem}[Single well case]
\label{theorem:single_well}
Let $f : \mathbb{R}^d \rightarrow \mathbb{R}$ be $\mathcal{C}^{2p}$ with $p \in \mathbb{N}$ and such that:
\begin{enumerate}
\item $f$ is coercive, i.e. $f(x) \longrightarrow + \infty$ when $||x|| \rightarrow + \infty$.
\item $\textup{argmin}(f) = {0}$.
\item $f \in \mathscr{A}_p^\star(0)$ and $f(0)=0$.
\item $e^{-f} \in L^1(\mathbb{R}^d)$.
\end{enumerate}
Let $(E_k)_k$, $(\alpha_i)_i$, $B$ and $g$ to be defined as in Algorithm \ref{algo:algorithm} stated right after, so that for all $h \in \mathbb{R}^d$,
\begin{equation*}
\frac{1}{t} \left[ f\left( x^\star + B \cdot (t^\alpha \ast h) \right) - f(x^\star) \right] \underset{t \rightarrow 0}{\longrightarrow} g(h) ,
\end{equation*}
and where $g$ is not constant in any of its variables. Moreover, assume that $g$ is coercive and the following technical hypothesis if $p \ge 5$:
\begin{align}
\label{equation:even_terms_null}
&\forall h \in \mathbb{R}^d, \ \forall (i_1,\ldots,i_p) \in \lbrace 0,2,\cdots,2p \rbrace^p, \\
& \frac{i_1}{2} + \cdots + \frac{i_p}{2p} < 1 \implies \ \nabla^{i_1+\cdots+i_p}f(x^\star) \cdot p_{_{E_1}}(h)^{\otimes i_1} \otimes \cdots \otimes p_{_{E_{p}}}(h)^{\otimes i_{p}} = 0 . \nonumber
\end{align}
Then the conclusion of Theorem \ref{theorem:athreya:1} holds, with:
$$ \left(\frac{1}{t^{\alpha_1}}, \ldots, \frac{1}{t^{\alpha_d}}\right) \ast (B^{-1} \cdot X_t) \overset{\mathscr{L}}{\longrightarrow} X  \ \text{ as } t \rightarrow 0,$$
where $X$ has a density proportional to $e^{-g(x)}$.
\end{theorem}

\begin{algorithm}
\label{algo:algorithm}
Let $f \in \mathscr{A}_p^\star(x^\star)$ for $p \in \mathbb{N}$.
\begin{enumerate}
\item Define $(F_k)_{0 \le k \le p-1}$ recursively as:
$$ \left\lbrace \begin{array}{l}
F_0 = \mathbb{R}^d \\
F_k = \lbrace h \in F_{k-1} : \ \forall h' \in F_{k-1}, \ \nabla^{2k} f(x^\star) \cdot h \otimes h'^{\otimes 2k-1} = 0 \rbrace.
\end{array} \right. $$

\item  For $1 \le k \le p-1$, define the subspace $E_k$ as the orthogonal complement of $F_k$ in $F_{k-1}$. By abuse of notation, define $E_p := F_{p-1}$.

\item Define $B \in \mathcal{O}_d(\mathbb{R})$ as an orthogonal transformation adapted to the decomposition
$$ \mathbb{R}^d = E_1 \oplus \cdots \oplus E_p .$$
\item Define for $1 \le i \le d$,
\begin{equation}
\alpha_i := \frac{1}{2j} \ \text{ for } i \in \lbrace \dim(E_1)+\cdots+\dim(E_{j-1})+1, \ldots, \dim(E_1)+\cdots+\dim(E_j) \rbrace .
\end{equation}
\item Define $g : \mathbb{R}^d \rightarrow \mathbb{R}$ as
\begin{equation}
\label{eq:def_g}
g(h) = \sum_{k=2}^{2p} \frac{1}{k!} \sum_{\substack{i_1,\ldots,i_p \in \lbrace 0,\ldots,k \rbrace \\ i_1 + \cdots + i_{p} = k \\ \frac{i_1}{2} + \cdots + \frac{i_p}{2p}=1 }} \binom{k}{i_1,\ldots,i_p} \nabla^k f(x^\star) \cdot p_{_{E_1}}(B \cdot h)^{\otimes i_1} \otimes \cdots \otimes p_{_{E_p}}(B \cdot h)^{\otimes i_p} .
\end{equation}
\end{enumerate}
\end{algorithm}

\textbf{Remarks :}
\begin{enumerate}
	\item The function $g$ is not unique, as we can choose any base $B$ adapted to the decomposition $\mathbb{R}^d = E_1 \oplus \cdots \oplus E_p$.
	\item The case $p \ge 5$ is fundamentally different from the case $p \le 4$, since Algorithm \ref{algo:algorithm} may fail to provide such $(E_k)_k$, $(\alpha_i)_i$, $B$ and $g$ if the technical hypothesis \eqref{equation:even_terms_null} is not fulfilled, as explained in Section \ref{section:order_10}. This yields fewer results for the case $p \ge 5$.
	\item For $p \in \lbrace 1,2,3,4 \rbrace$, the detail the expression of $g$ in \eqref{equation:order_2}, \eqref{equation:order_4}, \eqref{equation:order_6} and \eqref{equation:order_8:2} respectively.
	\item The function $g$ has the following general properties : $g$ is a non-negative polynomial of order $2p$; $g(0)=0$ and $\nabla g(0) = 0$.
	\item The condition on $g$ to be coercive may seem not natural. We give more details about the case where $g$ is not coercive in Section \ref{section:non_coercive} and give a way to deal with the simple generic case of non-coercivity. However dealing with the general case where $g$ is not coercive goes beyond the scope of our work.
	\item The hypothesis that $g$ is coercive is a necessary condition for $e^{-g} \in L^1(\mathbb{R}^d)$. We actually prove in Proposition \ref{prop:coercive} that it is a sufficient condition.
\end{enumerate}

\medskip

Still following \cite{athreya2010}, we study the multiple well case, i.e. the global minimum is attained in a finite number of points in $\mathbb{R}^d$, say $\lbrace x_1^\star,\ldots,x_m^\star \rbrace$ for some $m \in \mathbb{N}$. In this case, the limiting measure of $\pi_t$ will have its support in $\lbrace x_1^\star,\ldots,x_m^\star \rbrace$, with different weights.

\begin{theorem}[Athreya-Hwang, 2010]
\label{theorem:athreya:2}
Let $f:\mathbb{R}^d \rightarrow [0,\infty)$ measurable such that:
\begin{enumerate}
	\item $e^{-f} \in L^1(\mathbb{R}^d)$.
	\item For all $\delta > 0$, $\inf \lbrace f(x), \ ||x - x_i^\star|| > \delta, \ 1 \le i \le m \rbrace > 0$.
	\item There exist $(\alpha_{ij})_{\substack{1\le i \le m \\ 1 \le j \le d}}$ such that for all $i$, $j$, $\alpha_{ij} \ge 0$ and for all $i$:
	$$ \frac{1}{t} f(x_i^\star + (t^{\alpha_{i1}}h_1,\ldots,t^{\alpha_{id}}h_d)) \underset{t \rightarrow 0}{\longrightarrow} g_i(h_1,\ldots,h_d) \in [0,\infty) .$$
	\item For all $i \in \lbrace 1, \ldots, m \rbrace$,
	$$ \int_{\mathbb{R}^d} \sup_{0<t<1} e^{-\frac{f(x_i^\star + (t^{\alpha_{i1}}h_1,\ldots,t^{\alpha_{id}}h_d))}{t}}dh_1\ldots dh_d < \infty .$$
\end{enumerate}
Then, let $\alpha := \min_{1 \le i \le m} \left\lbrace \sum_{j=1}^d \alpha_{ij} \right\rbrace $ and let $J := \left\lbrace i \in \lbrace 1,\ldots,m \rbrace : \ \sum_{j=1}^d \alpha_{ij} = \alpha \right\rbrace$. For $0 < t < 1$, let $X_t$ be a random vector with distribution $\pi_t$. Then:
$$ X_t \overset{\mathscr{L}}{\underset{t \rightarrow 0}{\longrightarrow}} \frac{1}{\sum_{j \in J} \int_{\mathbb{R}^d} e^{-g_j(x)}dx} \sum_{i \in J} \int_{\mathbb{R}^d} e^{-g_i(x)}dx \cdot \delta_{x_i^\star} .$$
\end{theorem}

\begin{theorem}[Multiple well case]
\label{theorem:multiple_well}
Let $f : \mathbb{R}^d \rightarrow \mathbb{R}$ be $\mathcal{C}^{2p}$ for $p \in \mathbb{N}$ and such that:
\begin{enumerate}
	\item $f$ is coercive i.e. $f(x) \longrightarrow + \infty$ when $||x||\rightarrow +\infty$.
	\item $\argmin(f) = \lbrace x_1^\star, \ldots, x_m^\star \rbrace$ and for all $i$, $f(x_i^\star)=0$.
	\item For all $i \in \lbrace 1, \ldots, m \rbrace$, $f \in \mathscr{A}^\star_{p_i}(x_i^\star)$ for some $p_i \le p$.
	\item $e^{-f} \in L^1(\mathbb{R}^d)$.	
\end{enumerate}
Then, for every $i \in \lbrace 1, \ldots, m \rbrace$, we consider $(E_{ik})_k$, $(\alpha_{ij})_j$, $B_i$ and $g_i$ as defined in Algorithm \ref{algo:algorithm}, where we consider $f$ to be in $\mathscr{A}_{p_i}^\star(x_i^\star)$, so that for every $h \in \mathbb{R}^d$:
$$ \frac{1}{t} f(x_i^\star + B_i \cdot (t^{\alpha_i} \ast h)) \underset{t \rightarrow 0}{\longrightarrow} g_i(h_1,\ldots,h_d) \in [0,\infty) ,$$
where $t^{\alpha_i}$ is the vector $(t^{\alpha_{i1}},\ldots, t^{\alpha_{id}})$ and where $g_i$ is not constant in any of its variables. Furthermore, we assume that for all $i$, $g_i$ is coercive and the following technical hypothesis for every $i$ such that $p_i \ge 5$:
\begin{align*}
& \forall h \in \mathbb{R}^d, \ \forall (i_1,\ldots,i_{p_i}) \in \lbrace 0,2,\ldots,2{p_i} \rbrace^{p_i}, \\
& \frac{i_1}{2} + \cdots + \frac{i_{p_i}}{2p} < 1 \implies \ \nabla^{i_1+\cdots+i_{p_i}}f(x_i^\star) \cdot p_{_{E_{i1}}}(h)^{\otimes i_1} \otimes \cdots \otimes p_{_{E_{ip_i}}}(h)^{\otimes i_{p_i}} = 0 . \nonumber
\end{align*}
Let $\alpha := \min_{1 \le i \le m} \left\lbrace \sum_{j=1}^d \alpha_{ij} \right\rbrace $ and let $J := \left\lbrace i \in \lbrace 1,\ldots,m \rbrace : \ \sum_{j=1}^d \alpha_{ij} = \alpha \right\rbrace$. Then:
$$ X_t \underset{t \rightarrow 0}{\longrightarrow} \frac{1}{\sum_{j \in J} \int_{\mathbb{R}^d} e^{-g_j(x)}dx} \sum_{i \in J} \int_{\mathbb{R}^d} e^{-g_i(x)}dx \cdot \delta_{x_i^\star} .$$
Moreover, let $\delta > 0$ be small enough so that the balls $\mathcal{B}(x_i^\star,\delta)$ are disjoint, and define the random vector $X_{it}$ to have the law of $X_t$ conditionally to the event $||X_t - x_i^\star||< \delta$. Then:
$$\left(\frac{1}{t^{\alpha_{i1}}}, \ldots,\frac{1}{t^{\alpha_{id}}}\right) \ast (B_i^{-1} \cdot X_{it}) \overset{\mathscr{L}}{\longrightarrow} X_i  \ \text{ as } t \rightarrow 0 ,$$
where $X_i$ has a density proportional to $e^{-g_i(x)}$.
\end{theorem}

\section{Expansion of $f$ at a local minimum with degenerate derivatives}
\label{section:expansion}

In this section, we aim at answering to the problem stated in \eqref{eq:alpha_developpement} in order to devise conditions to apply Theorem \ref{theorem:athreya:1}. This problem can also be considered in a more general setting, independently of the study of the convergence of Gibbs measures. It provides a non degenerate higher order nested expansion of $f$ at a local minimum when some of the derivatives of $f$ are degenerate. Note here that we only need $x^\star$ to be a local minimum instead of a global minimum, since we only give local properties.

For $k \le p$, we define the tensor of order $k$, $T_k := \nabla^k f(x^\star)$.

\subsection{Expansion of $f$ for any order $p$}
\label{sec:expansion_any_order}

In this section, we state our result in a synthetic form. The proofs of the cases $p =1,2,3,4$ are individually detailled in Sections \ref{section:order_2}, \ref{section:order_4}, \ref{section:order_6} and \ref{section:order_8} respectively.

\begin{theorem}
\label{theorem:main}
Let $f : \mathbb{R}^d \rightarrow \mathbb{R}$ be $\mathcal{C}^{2p}$ for some $p \in \mathbb{N}$ and assume that $f \in \mathscr{A}_p^\star(x^\star)$ for some $x^\star \in \mathbb{R}^d$.
\begin{enumerate}
\item If $p \in \lbrace 1,2,3,4 \rbrace$, then there exists orthogonal subspaces of $\mathbb{R}^d$, $E_1, \ \ldots, \ E_{p}$ such that
$$ \mathbb{R}^d = E_1 \oplus \cdots \oplus E_{p},$$
and satisfying for every $h \in \mathbb{R}^d$:
\begin{align}
\label{equation:order_p:1}
& \frac{1}{t} \left[ f\left(x^\star + t^{1/2}p_{_{E_1}}(h) + \cdots + t^{1/(2p)}p_{_{E_{p}}}(h) \right) - f(x^\star) \right] \\
\label{equation:order_p:2}
\underset{t \rightarrow 0}{\longrightarrow} & \sum_{k=2}^{2p} \frac{1}{k!} \sum_{\substack{i_1,\ldots,i_p \in \lbrace 0,\cdots,k \rbrace \\ i_1 + \cdots + i_{p} = k \\ \frac{i_1}{2} + \cdots + \frac{i_p}{2p}=1 }} \binom{k}{i_1,\ldots,i_{p}} T_k \cdot p_{_{E_1}}(h)^{\otimes i_1} \otimes \cdots \otimes p_{_{E_{p}}}(h)^{\otimes i_{p}}. 
\end{align}
The convergence is uniform with respect to $h$ on every compact set. Moreover, let $B \in \mathcal{O}_d(\mathbb{R})$ be an orthogonal transformation adapted to the decomposition $E_1 \oplus \cdots \oplus E_{p}$, then
\begin{equation}
\label{equation:order_p:3}
\frac{1}{t} \left[ f\left( x^\star + B \cdot (t^\alpha \ast h) \right) - f(x^\star) \right] \underset{t \rightarrow 0}{\longrightarrow} g(h) ,
\end{equation}
where 
\begin{equation}
\label{eq:def_g:2}
g(h) = \sum_{k=2}^{2p} \frac{1}{k!} \sum_{\substack{i_1,\ldots,i_p \in \lbrace 0,\ldots,k \rbrace \\ i_1 + \cdots + i_{p} = k \\ \frac{i_1}{2} + \cdots + \frac{i_p}{2p}=1 }} \binom{k}{i_1,\ldots,i_{p}} T_k \cdot p_{_{E_1}}(B \cdot h)^{\otimes i_1} \otimes \cdots \otimes p_{_{E_{p}}}(B \cdot h)^{\otimes i_p}
\end{equation}
is not constant in any of its variables $h_1, \ \ldots, \ h_d$ and
\begin{align}
\label{eq:def_alpha:2}
\alpha_i & := \frac{1}{2j} \ \text{ for } i \in \lbrace \dim(E_1)+\cdots+\dim(E_{j-1})+1, \ldots, \dim(E_1)+\cdots+\dim(E_j) \rbrace .
\end{align}

\item If $p \ge 5$ and if there exist orthogonal subspaces of $\mathbb{R}^d$, $E_1, \ \ldots, \ E_{p}$ such that
$$ \mathbb{R}^d = E_1 \oplus \cdots \oplus E_{p}$$
and satisfying the following additional assumption
\begin{align}
\label{equation:even_terms_null:2}
& \forall h \in \mathbb{R}^d, \ \forall (i_1,\ldots,i_p) \in \lbrace 0,2,\ldots,2p \rbrace^{p}, \\
& \frac{i_1}{2} + \cdots + \frac{i_p}{2p} < 1 \implies \ T_{i_1+\cdots+i_p} \cdot p_{_{E_1}}(h)^{\otimes i_1} \otimes \cdots \otimes p_{_{E_{p}}}(h)^{\otimes i_{p}} = 0 , \nonumber
\end{align}
then \eqref{equation:order_p:2} stills holds true, as well as the uniform convergence on every compact set. Moreover, if $B \in \mathcal{O}_d(\mathbb{R})$ is an orthogonal transformation adapted to the previous decomposition, then \eqref{equation:order_p:3} still hold true. However, depending on the function $f$, such subspaces do not necessarily exist.
\end{enumerate}

\end{theorem}

\medskip

\textbf{Remarks:}
\begin{enumerate}
	\item The limit \eqref{equation:order_p:2} can be rewritten as:
	$$\sum_{k=2}^{2p} \sum_{\substack{i_1,\ldots,i_p \in \lbrace 0,\cdots,k \rbrace \\ i_1 + \cdots + i_{p} = k \\ \frac{i_1}{2} + \cdots + \frac{i_p}{2p}=1 }} T_k \cdot \frac{p_{_{E_1}}(h)^{\otimes i_1}}{i_1!} \otimes \cdots \otimes \frac{p_{_{E_{p}}}(h)^{\otimes i_{p}}}{i_p!}. $$
	\item For $p \in \lbrace 1,2,3,4 \rbrace$, we explicitly give the expression of the sum \eqref{equation:order_p:2} and the $p$-tuples $(i_1,\ldots,i_p)$ such that $\frac{i_1}{2}+\cdots+\frac{i_p}{2p}=1$, in \eqref{equation:order_2}, \eqref{equation:order_4}, \eqref{equation:order_6} and \eqref{equation:order_8:2} respectively.
	\item For $p \in \lbrace 1,2,3,4 \rbrace$, we give in Algorithm \ref{algo:algorithm} an explicit construction of the orthogonal subspaces $E_1, \ \ldots, \ E_{p}$ as complementaries of annulation sets of some derivatives of $f$.
	\item The case $p \ge 5$ is fundamentally different from the case $p \in \lbrace 1,2,3,4 \rbrace$. The strategy of proof developed for $p \in \lbrace 1,2,3,4 \rbrace$ fails if the assumption \eqref{equation:even_terms_null:2} is not satisfied. In \ref{section:order_10} a counter-example is detailed. The case $p \ge 5$ yields fewer results than for $p \le 4$, as the assumption \eqref{equation:even_terms_null:2} is strong.
	\item For $p \ge 5$, such subspaces $E_1$, $\ldots$, $E_p$ may also be obtained from Algorithm \ref{algo:algorithm}, however \eqref{equation:even_terms_null:2} is not necessarily true in this case.
\end{enumerate}

\medskip

The proof of Theorem \ref{theorem:main} is given first individually for each $p \in \lbrace 1,2,3,4 \rbrace$, in Sections \ref{section:order_2}, \ref{section:order_4}, \ref{section:order_6}, \ref{section:order_8} respectively. The proof for $p \ge 5$ is given in Section \ref{section:order_10}. The proof of the uniform convergence and of the fact that $g$ is not constant is given in Section \ref{subsec:unif_non_constant}.

\subsection{Review of the one dimensional case}

We review the case $d=1$, as it guides us for the proof in the case $d \ge 2$. The strategy is to find the first derivative $f^{(m)}(x^\star)$ which is non zero and then to choose $\alpha_1 = 1/m$.

\begin{proposition}
Let $f : \mathbb{R} \rightarrow \mathbb{R}$ be $\mathcal{C}^p$ for some $p \in \mathbb{N}$ and let $x^\star$ be a strict polynomial local minimum of $f$. Then :
\begin{enumerate}
\item The order of the local minimum $m$ is an even number and $f^{(m)}(x^\star) > 0$.
\item $ f(x^\star + h) \underset{h \rightarrow 0}{=} f(x^\star) + \frac{f^{(m)}(x^\star)}{m!}h^p + o(h^m) $
\end{enumerate}
\end{proposition}
Then $\alpha_1 := 1/m$ is the solution of \eqref{eq:alpha_developpement} and
$$\frac{1}{t} (f(x^\star + t^{1/m}h) - f(x^\star)) \underset{t \rightarrow 0}{\longrightarrow} \frac{f^{(m)}(x^\star)}{m!} h^m $$
which is a non-constant function of $h$, since $f^{(m)}(x^\star) \ne 0$.
The direct proof using the Taylor formula is left to the reader.

\subsection{Proof of Theorem \ref{theorem:main} for $p=1$}
\label{section:order_2}

Let $f \in \mathscr{A}^\star_1(x^\star)$. The assumption that $x^\star$ is a strict polynomial local minimum at order $2$ implies that $\nabla^2 f(x^\star)$ is positive definite. Let us denote $(\beta_i)_{1 \le i \le d}$ its positive eigenvalues. By the spectral theorem, let us write $\nabla^2 f(x^\star ) = B {\rm Diag}(\beta_{1:d})B^T$ for some $B \in \mathcal{O}_d(\mathbb{R})$. Then:
\begin{equation}
\label{equation:order_2}
\frac{1}{t} (f(x^\star + t^{1/2}B \cdot h) - f(x^\star)) \underset{t \rightarrow 0}{\longrightarrow} \frac{1}{2}\sum_{i=1}^{d} \beta_i h_i^2.
\end{equation}
Thus, a solution of \eqref{eq:alpha_developpement} is $\alpha_1=\cdots=\alpha_d=\frac{1}{2}$, $B$, and $g(h_1,\ldots,h_d)=\frac{1}{2} \sum_{i=1}^d \beta_i h_i^2$, which is a non-constant function of every $h_1, \ \ldots, \ h_d$, since for all $i$, $\beta_i$ is positive.

In the following, our objective is to establish a similar result when $\nabla^2 f(x^\star)$ is not necessarily positive definite.

\subsection{Proof of Theorem \ref{theorem:main} for $p=2$}
\label{section:order_4}

\begin{theorem}
\label{theorem:order_4}
Let $f \in \mathscr{A}_2(x^\star)$. Then there exist orthogonal subspaces $E$ and $F$ such that $\mathbb{R}^d = E \oplus F$, and that for all $h \in \mathbb{R}^d$:
\begin{align}
\frac{1}{t} & \left[ f(x^\star + t^{1/2}p_{_{E}}(h) + t^{1/4}p_{_F}(h)) - f(x^\star)\right] \nonumber \\
\label{equation:order_4}
\underset{t \rightarrow 0}{\longrightarrow} & \ \frac{1}{2} \nabla^2 f(x^\star) \cdot p_{_E}(h)^{\otimes 2} + \frac{1}{2} \nabla^3 f(x^\star) \cdot p_{_E}(h) \otimes p_{_F}(h)^{\otimes 2} + \frac{1}{4!} \nabla^4 f(x^\star)\cdot p_{_F}(h)^{\otimes 4} .
\end{align}
Moreover, if $f \in \mathscr{A}_2^\star (x^\star)$, then this is a solution to the problem \eqref{eq:alpha_developpement}, with $E_1 = E$, $E_2=F$, $\alpha$ defined in \eqref{eq:def_alpha:2}, $B$ adapted to the previous decomposition and $g$ defined in \eqref{eq:def_g:2}.
\end{theorem}

\noindent \textbf{Remark:} The set of $2$-tuples $(i_1,i_2)$ such that $\frac{i_1}{2} + \frac{i_2}{4} = 1$, are $(2,0)$, $(1,2)$ and $(0,4)$, which gives the terms appearing in the sum in \eqref{equation:order_p:2}.

\begin{proof}
Let $F := \{ h \in \mathbb{R}^d : \ \nabla^2 f(x^\star) \cdot h^{\otimes 2} = 0 \}$. By the spectral theorem and since $\nabla^2 f(x^\star)$ is positive semi-definite, $F = \{ h \in \mathbb{R}^d : \ \nabla^2 f(x^\star) \cdot h = 0^{\otimes 1} \}$ is a vector subspace of $\mathbb{R}^d$. Let $E$ be the orthogonal complement of $F$ in $\mathbb{R}^d$.

For $h \in \mathbb{R}^d$ we expand the left term of \eqref{equation:order_4} using the Taylor formula up to order $4$ and the multinomial formula \eqref{equation:multinomial}, giving
$$ \sum_{k=2}^4 \frac{1}{k!} \sum_{\substack{i_1,i_2 \in \lbrace 0,\ldots,k \rbrace \\ i_1+i_2=k}} \binom{k}{i_1,i_2} t^{\frac{i_1}{2} + \frac{i_2}{4}-1} T_k \cdot p_{_E}(h)^{\otimes i_1} \otimes p_{_F}(h)^{\otimes i_2} + o(1).$$
The terms with coefficient $t^a$, $a>0$, are $o(1)$ as $t \rightarrow 0$. 
By definition of $F$ we have $\nabla^2 f(x^\star) \cdot p_{_F}(h) = 0^{\otimes 1}$, so we also have
$$\nabla ^3 f(x^\star) \cdot p_{_F}(h)^{\otimes 3} = 0$$ by the local minimum condition. This yields the convergence stated in \eqref{equation:order_4}.

Moreover, if $x^\star$ is a local minimum of polynomial order 4, then by the local minimum condition, $\nabla^4 f(x^\star) > 0$ on $F$ in the sense of \eqref{eq:tensor_positive_def}. Moreover, since $\nabla^2 f(x^\star) > 0$ on $E$, then the limit is not constant in any $h_1, \ \ldots, \ h_d$.
\end{proof}

\noindent \textbf{Remark:} The cross odd term is not necessarily null. For example, consider
$$ \begin{array}{rrl}
f : & \mathbb{R}^2 & \longrightarrow \mathbb{R} \\
& (x,y) & \longmapsto x^2 + y^4 + xy^2.
\end{array} $$
Then $f$ admits a global minimum at $x^\star=0$ since $|xy^2| \le \frac{1}{2}(x^2+y^4)$. We have $E_1 = \mathbb{R}(1,0)$, $E_2 = \mathbb{R}(0,1)$ and for all $(x,y) \in \mathbb{R}^2$, $T_3 \cdot (xe_1) \otimes (ye_2)^{\otimes 2} = 2xy^2$ is not identically null.

\subsection{Difficulties beyond the 4th order and Hilbert's $17^{\text{th}}$ problem}
\label{section:hilbert}

If we do not assume as in the previous section that $\nabla^4 f(x^\star)$ is not positive on $F$, then we carry on the development of $f(x^\star + h)$ up to higher orders.
A first idea is to consider $F_2 := \{ h \in F: \ \nabla^4 f(x^\star) \cdot h^{\otimes 4}=0 \} \subseteq F$ and $E_2$ a complement subspace of $F_2$ in $F$, and to continue this process by induction as in Section \ref{section:order_4}. However, $F_2$ is not necessarily a subspace of $F$.

Indeed, let $T$ be a symmetric tensor defined on $\mathbb{R}^{d'}$ of order $2k$ with $k \in \mathbb{N}$. As $T$ is symmetric, there exist vectors $v^1, \ \ldots, \ v^q \in \mathbb{R}^{d'}$, and scalars $\lambda_1, \ \ldots, \ \lambda_q \in \mathbb{R}$ such that $T = \sum_i \lambda_i (v^i)^{\otimes 2k}$ (see \cite{comon2008}, Lemma 4.2.), so
$$ \forall h \in \mathbb{R}^{d'}, \ T \cdot h^{\otimes 2k} = \sum_{i=1}^q \lambda_i (v^i)^{\otimes 2k} \cdot h^{\otimes 2k} = \sum_{i=1}^q \lambda_i \langle v^i , h \rangle^{2k} .$$
For $k = 2$ and $T = \nabla^{2k} f(x^\star)_{|_F}$, since $x^\star$ is a local minimum, we have, identifying $F$ and $\mathbb{R}^{d'}$,
$$ \forall h \in \mathbb{R}^{d'}, \ T \cdot h^{\otimes 2k} \ge 0 $$
Then, we could think it implies that for all $i$, $\lambda_i \ge 0$, and then
$$ T \cdot h^{\otimes 2k} = 0 \ \implies \forall i, \ \langle v^i, h \rangle = 0$$
which would give a linear caracterization of $\{h \in \mathbb{R}^{d'} : \ T \cdot h^{\otimes 2k} = 0 \}$ and in this case, $F_2$ would be a subspace of $F$.
However this reasoning is not correct in general as we do not have necessarily that for all $i$, $\lambda_i \ge 0$.

We can build counter-examples as follows. Since $T$ is a non-negative symmetric tensor, $T$ can be seen as a non-negative homogeneous polynomial of degree $2k$ with $d'$ variables. A counter-example at order $2k=4$ is $T(X,Y,Z) = ((X-Y)(X-Z))^2$, which is a non-negative polynomial of order 4, but $\{T=0\} = \{ X=Y \text{ or } X=Z \}$, which is not a vector space.

Another counterexample given in \cite{motzkin1967} at order $2k = 6$ is the following. We define
$$ T(X,Y,Z) = Z^6 + X^4 Y^2 + X^2 Y^4 - 3 X^2 Y^2 Z^2 $$
By the arithmetic-geometric mean inequality and its equality case, $T$ is non-negative and $T(x,y,z) = 0$ if and only if $z^6 = x^4 y^2 = x^2 y^4 $, so that
$$\{ T = 0 \} =
\mathbb{R} \begin{pmatrix}
1 \\ 
1 \\ 
1
\end{pmatrix} \
\cup \ \mathbb{R} \begin{pmatrix}
- 1 \\ 
1 \\ 
1
\end{pmatrix} \ 
\cup \ \mathbb{R} \begin{pmatrix}
1 \\ 
- 1 \\ 
1
\end{pmatrix} \
\cup \ \mathbb{R} \begin{pmatrix}
1 \\ 
1 \\ 
- 1
\end{pmatrix}. $$
Hence, $\{ T = 0 \}$ is not a subspace of $\mathbb{R}^{3}$. In particular $T$ cannot be written as $\sum_i \lambda_i (v^i)^{\otimes 2k}$ with $\lambda_i \ge 0$.

In fact, this problem is linked with the Hilbert's seventeenth problem that we recall below.
\begin{problem}[Hilbert's seventeeth problem]
Let $P$ be a non-negative polynomial with $d'$ variables, homogeneous of even degree $2k$. Find polynomials $P_1, \ \ldots, \ P_r$ with $d'$ variables, homogeneous of degree $k$, such that $P = \sum_{i=1}^r P_i^2$
\end{problem}

Hilbert proved in 1888 \cite{hilbert1888} that there does not always exist a solution. In general $\{T=0\}$ is not even a submanifold of $\mathbb{R}^{d'}$. Indeed, taking $T :h \mapsto \nabla^{2k} f(x^\star) \cdot h^{\otimes 2k}$, we have $\partial_h T\cdot h = 2k \nabla^{2k} f(x^\star) \cdot h^{\otimes 2k-1} $ is not surjective in $h=0$, so the surjectivity condition for $\lbrace T=0 \rbrace$ to be a submanifold is not fulfilled.

\subsection{Proof of Theorem \ref{theorem:main} for $p=3$}
\label{section:order_6}

We slightly change our strategy of proof developed in Section \ref{section:order_4}. For $k \ge 2$, we define $F_k$ recursively as
\begin{equation}
\label{eq:F_k_def}
F_k := \lbrace h \in F_{k-1} : \ \forall h' \in F_{k-1}, \ \nabla^{2k} f(x^\star) \cdot h \otimes h'^{\otimes 2k-1} = 0 \rbrace,
\end{equation}
instead of $\lbrace h \in F_{k-1} : \ \nabla^{2k} f(x^\star) \cdot h^{\otimes 2k} = 0 \rbrace$. Then, by construction, $F_k$ is a vector subspace of $\mathbb{R}^d$.

\begin{theorem}
\label{theorem:order_6}
Let $f \in \mathscr{A}_3(x^\star)$. Then there exist orthogonal subspaces of $\mathbb{R}^d$, $E_1$, $E_2$ and $F_2$, such that
$$ \mathbb{R}^d = E_1 \oplus E_2 \oplus F_2 ,$$
and such that for all $h \in \mathbb{R}^d$,
\begin{align}
\label{equation:order_6}
\frac{1}{t} & \left[ f(x^\star + t^{1/2}p_{_{E_1}}(h) + t^{1/4}p_{_{E_2}}(h) + t^{1/6}p_{_{F_2}}(h)) - f(x^\star) \right] \\
\underset{t \rightarrow 0}{\longrightarrow} & \ \frac{1}{2} \nabla^2 f(x^\star) \cdot p_{_{E_1}}(h)^{\otimes 2} + \frac{1}{2} \nabla^3 f(x^\star) \cdot p_{_{E_1}}(h) \otimes p_{_{E_2}}(h)^{\otimes 2} + \frac{1}{4!} \nabla^4 f(x^\star)\cdot p_{_{E_2}}(h)^{\otimes 4} \nonumber \\
& + \frac{4}{4!}\nabla^4 f(x^\star)\cdot p_{_{E_1}}(h)\otimes p_{_{F_2}}(h)^{\otimes 3} + \frac{10}{5!}\nabla^5 f(x^\star) \cdot p_{_{E_2}}(h)^{\otimes 2} \otimes p_{_{F_2}}(h)^{\otimes 3} + \frac{1}{6!}\nabla^6 f(x^\star) \cdot p_{_{F_2}}(h)^{\otimes 6}. \nonumber
\end{align}
Moreover, if $f \in \mathscr{A}_3^\star(x^\star)$, then this is a solution to the problem \eqref{eq:alpha_developpement}, with $E_3 = F_2$, $\alpha$ defined in \eqref{eq:def_alpha:2}, $B$ adapted to the previous decomposition and $g$ defined in \eqref{eq:def_g:2}.
\end{theorem}

\noindent \textbf{Remark:} The set of $3$-tuples $(i_1,i_2,i_3)$ such that $\frac{i_1}{2} + \frac{i_2}{4} + \frac{i_3}{6} = 1$, are $(2,0,0)$, $(1,2,0)$, $(0,4,0)$, $(1,0,3)$, $(0,2,3)$, $(0,0,6)$,  which gives the terms appearing in \eqref{equation:order_p:2}.

\begin{proof}
We consider the subspace
$$ F_1 := \lbrace h \in \mathbb{R}^d : \ T_2 \cdot h^{\otimes 2} = 0 \rbrace  = \lbrace h \in \mathbb{R}^d : \ T_2 \cdot h = 0^{\otimes 1} \rbrace, $$
since $T_2 \ge 0$.
Then, let $E_1$ be the orthogonal complement of $F_1$ in $\mathbb{R}^d$ and consider the vector subspace of $F_1$ defined by
$$ F_2 = \lbrace h \in F_1 : \ \forall h' \in F_1, \ T_4 \cdot h \otimes h'^{\otimes 3} = 0 \rbrace .$$
Let $E_2$ be the orthogonal complement of $F_2$ in $F_1$. Then we have
$$ \mathbb{R}^d = E_1 \oplus F_1 = E_1 \oplus E_2 \oplus F_2 .$$
For $h \in \mathbb{R}^d$ we expand the left term of \eqref{equation:order_6} using the Taylor formula up to order $6$ and the multinomial formula \eqref{equation:multinomial}, giving
$$ \sum_{k=2}^6 \frac{1}{k!} \sum_{\substack{i_1,i_2,i_3 \in \lbrace 0,\ldots,k \rbrace \\ i_1+i_2+i_3=k}} \binom{k}{i_1,i_2,i_3} t^{\frac{i_1}{2} + \frac{i_2}{4} + \frac{i_3}{6}-1} T_k \cdot p_{_{E_1}}(h)^{\otimes i_1} \otimes p_{_{E_2}}(h)^{\otimes i_2} \otimes p_{_{F_2}}(h)^{\otimes i_3} + o(1),$$
and we prove the convergence stated in \eqref{equation:order_6}.

\medskip

All the terms with coefficient $t^a$ where $a>0$ are $o(1)$ as $t \rightarrow 0$.

\textbf{Order 2:} we have $T_2 \cdot p_{_{E_2}}(h) = 0^{\otimes 1}$ and $T_2 \cdot p_{_{F_2}}(h) = 0^{\otimes 1}$ so the only term for $k=2$ is $\frac{1}{2} T_2 \cdot p_{_{E_1}} (h)^{\otimes 2}$.

\medskip

\textbf{Order 3:} $\triangleright$ Since $x^\star$ is a local minimum and $T_2 \cdot p_{_{F_1}}(h)^{\otimes 2} = 0$, we have $T_3 \cdot p_{_{F_1}}(h)^{\otimes 3} = 0$. Then, using property Proposition \ref{proposition:null_tensor:1}, if the factor $p_{_{E_1}}(h)$ does not appear as an argument in $T_3$, then the corresponding term is zero.

$\triangleright$ Let us prove that
\begin{equation*}
T_3 \cdot p_{_{E_1}}(h) \otimes p_{_{F_2}}(h)^{\otimes 2} = 0 .
\end{equation*}
Using Theorem \ref{theorem:order_4} with $E = E_1$, $F = E_2 \oplus F_2$, we have in particular that for all $h \in \mathbb{R}^d$,
\begin{equation}
\label{equation:order_4_positive}
\frac{1}{2} T_2 \cdot p_{_E}(h)^{\otimes 2} + \frac{1}{2} T_3 \cdot p_{_E}(h) \otimes p_{_F}(h)^{\otimes 2} + \frac{1}{4!} T_4 \cdot p_{_F}(h)^{\otimes 4} \ge 0 .
\end{equation}
Then taking $h \in E_1 \oplus F_2$ so that $h= p_{_{E_1}}(h)+p_{_{F_2}}(h)$ and with
\begin{equation}
\label{eq:proof:T4_null}
\left[T_4\cdot p_{_{F_2}}(h)\right]_{|F_1} \equiv 0^{\otimes 3},
\end{equation}
we may rewrite \eqref{equation:order_4_positive} as
$$ \frac{1}{2} T_2 \cdot p_{_{E_1}}(h)^{\otimes 2} + \frac{1}{2} T_3 \cdot p_{_{E_1}}(h) \otimes p_{_{F_2}}(h)^{\otimes 2} \ge 0 .$$
Now, considering $h' = \lambda h$, we have that for all $\lambda \in \mathbb{R}$, 
$$ \lambda^2 \left(\frac{1}{2} T_2 \cdot p_{_{E_1}}(h)^{\otimes 2} + \frac{\lambda}{2} T_3 \cdot p_{_{E_1}}(h) \otimes p_{_{F_2}}(h)^{\otimes 2} \right) \ge 0 ,$$
so that necessarily $T_3 \cdot p_{_{E_1}}(h) \otimes p_{_{F_2}}(h)^{\otimes 2} = 0$.

$\triangleright$ Let us prove that
$$ T_3 \cdot p_{_{E_1}}(h) \otimes p_{_{E_2}}(h) \otimes p_{_{F_2}}(h) = 0 .$$
We use again \eqref{equation:order_4_positive}, with $p_{_F}(h) = p_{_{E_2}}(h) + p_{_{F_2}}(h)$, so that
$$ \frac{1}{2} T_2 \cdot p_{_{E_1}}(h)^{\otimes 2} + \frac{1}{2} T_3 \cdot p_{_{E_1}}(h) \otimes \left(p_{_{E_2}}(h) + p_{_{F_2}}(h)\right)^{\otimes 2} + \frac{1}{4!} T_4 \cdot \left(p_{_{E_2}}(h) + p_{_{F_2}}(h)\right)^{\otimes 4} \ge 0 . $$
But using \eqref{eq:proof:T4_null} and that $T_3 \cdot p_{_{E_1}}(h) \otimes p_{_{F_2}}(h)^{\otimes 2} = 0$, we obtain
$$ \frac{1}{2} T_2 \cdot p_{_{E_1}}(h)^{\otimes 2} + \frac{1}{2} T_3 \cdot p_{_{E_1}}(h) \otimes p_{_{E_2}}(h)^{\otimes 2} + T_3 \cdot p_{_{E_1}}(h) \otimes p_{_{E_2}}(h) \otimes p_{_{F_2}}(h) + \frac{1}{4!} T_4 \cdot p_{_{E_2}}(h)^{\otimes 4} \ge 0 . $$
Now, considering $h' = p_{_{E_1}}(h) + p_{_{E_2}}(h) + \lambda p_{_{F_2}}(h)$, we have that for all $\lambda \in \mathbb{R}$,
$$ \frac{1}{2} T_2 \cdot p_{_{E_1}}(h)^{\otimes 2} + \frac{1}{2} T_3 \cdot p_{_{E_1}}(h) \otimes p_{_{E_2}}(h)^{\otimes 2} + \lambda T_3 \cdot p_{_{E_1}}(h) \otimes p_{_{E_2}}(h) \otimes p_{_{F_2}}(h) + \frac{1}{4!} T_4 \cdot p_{_{E_2}}(h)^{\otimes 4} \ge 0 ,$$
so necessarily $T_3 \cdot p_{_{E_1}}(h) \otimes p_{_{E_2}}(h) \otimes p_{_{F_2}}(h) = 0$.

$\triangleright$ The last remaining term for $k=3$ is $\frac{1}{2} T_3 \cdot p_{_{E_1}}(h) \otimes p_{_{E_2}}(h)^{\otimes 2}$.

\medskip

\textbf{Order 4:} If the factor $p_{_{E_1}}(h)$ does not appear and if the factor $p_{_{F_2}}(h)$ appears at least once, then using \eqref{eq:proof:T4_null} the corresponding term is zero. If $p_{_{E_1}}(h)$ appears, the only term with a non-positive exponent of $t$ is $\frac{4}{4!} T_4 \cdot p_{_{E_1}}(h) \otimes p_{_{F_2}}(h)^{\otimes 3}$. So the only terms for $k=4$ are $\frac{1}{4!}T_4 \cdot p_{_{E_2}}(h)^{\otimes 4}$ and $\frac{4}{4!} T_4 \cdot p_{_{E_1}}(h)\otimes p_{_{F_2}}(h)^{\otimes 3}$.

\medskip

\textbf{Order 5:} $\triangleright$ The terms where $p_{_{E_1}}(h)$ appears at least once have a coefficient $t^a$ with $a>0$ so are $o(1)$ when $t \rightarrow 0$.

$\triangleright$ We have $T_2 \cdot p_{_{F_2}}(h)^{\otimes 2} = 0$, $T_3 \cdot p_{_{F_2}}(h)^{\otimes 3} = 0$, $T_4 \cdot p_{_{F_2}}(h)^{\otimes 4} = 0$ and since $x^\star$ is a local minimum, we have
$$ T_5 \cdot p_{_{F_2}}(h)^{\otimes 5} = 0 .$$

$\triangleright$ Let us prove that
$$ T_5 \cdot p_{_{E_2}}(h) \otimes p_{_{F_2}}(h)^{\otimes 4} = 0 .$$
Let $h \in \mathbb{R}^d$. We have
$$ \frac{1}{t^{11/12}} \left[ f(x^\star + t^{1/4}p_{_{E_2}}(h) + t^{1/6}p_{_{F_2}}(h)) - f(x^\star) \right] \underset{t \rightarrow 0}{\longrightarrow} \frac{1}{4!} T_5 \cdot p_{_{E_2}}(h) \otimes p_{_{F_2}}(h)^{\otimes 4} \ge 0 .$$
Hence, considering $h'=\lambda h$, we have for every $\lambda \in \mathbb{R}$,
$$ \lambda^5 T_5 \cdot p_{_{E_2}}(h) \otimes p_{_{F_2}}(h)^{\otimes 4} \ge 0 ,$$
which yields the desired result.

$\triangleright$ The only remaining term for $p=5$ is
$$\frac{10}{5!} T_5 \cdot p_{_{E_2}}(h)^{\otimes 2} \otimes p_{_{F_2}}(h)^{\otimes 3} .$$

\medskip

\textbf{Order 6:} The only term for $k=6$ is $\frac{1}{6!} T_6 \cdot p_{_{F_2}}(h)^{\otimes 6}$ ; the other terms have a coefficient $t^a$ with $a > 0$, so are $o(1)$ when $t \rightarrow 0$.
\end{proof}

\textbf{Remark :} As in Theorem \ref{theorem:order_4} and the remark that follows, the remaining odd cross-terms cannot be proved to be zero using the same method of proof, and may be actually not zero. For example, consider:
$$ \begin{array}{rrl}
f : & \mathbb{R}^2 & \longrightarrow \mathbb{R} \\
& (x,y) & \longmapsto x^4 + y^6 + x^2y^3,
\end{array} $$
which satisfies $h \mapsto \nabla^5 f(x^\star) \cdot p_{_{E_2}}(h)^{\otimes 2} \otimes p_{_{F_2}}(h)^{\otimes 3} \not\equiv 0$.

\subsection{Proof of Theorem \ref{theorem:main} for $p=4$}
\label{section:order_8}

\begin{theorem}
\label{theorem:order_8}
Let $f \in \mathscr{A}_4(x^\star)$. Then there exist orthogonal subspaces of $\mathbb{R}^d$, $E_1$, $E_2$, $E_3$ and $F_3$ such that
$$ \mathbb{R}^d = E_1 \oplus E_2 \oplus E_3 \oplus F_3 ,$$
and for all $h \in \mathbb{R}^d$,
\begin{align}
& \frac{1}{t} \left[ f(x^\star + t^{1/2}p_{_{E_1}}(h) + t^{1/4}p_{_{E_2}}(h) + t^{1/6}p_{_{E_3}}(h) + t^{1/8}p_{_{F_3}}(h)) - f(x^\star) \right] \nonumber \\
\label{equation:order_8:2}
\underset{t \rightarrow 0}{\longrightarrow} \ & \sum_{k=2}^{8} \frac{1}{k!} \sum_{\substack{i_1,\ldots,i_{4} \in \lbrace 0,\ldots,k \rbrace \\ i_1 + \cdots + i_{4} = k }} \binom{k}{i_1,\ldots,i_{4}} T_k \cdot p_{_{E_1}}(h)^{\otimes i_1} \otimes p_{_{E_2}}(h)^{\otimes i_2} \otimes p_{_{E_3}}(h)^{\otimes i_3} \otimes p_{_{F_3}}(h)^{\otimes i_4}.
%
\end{align}
These terms are summarized as tuples $(i_1,\ldots,i_4)$ in Table \ref{figure:terms_exponent_1}. Moreover, if $f \in \mathscr{A}_4^\star(x^\star)$, then this is a solution to \eqref{eq:alpha_developpement}, with $E_4 = F_3$, $\alpha$ defined in \eqref{eq:def_alpha:2}, $B$ adapted to the previous decomposition and $g$ defined in \eqref{eq:def_g:2}.
\end{theorem}

\begin{table}
	\centering
	\begin{tabular}{|c|c|}
	\hline 
	\rule[-1ex]{0pt}{2.5ex} Order $2$ & $(2,0,0,0)$ \\ 
	\hline 
	\rule[-1ex]{0pt}{2.5ex} Order $3$ & $(2,1,0,0)$ \\ 
	\hline 
	\rule[-1ex]{0pt}{2.5ex} Order $4$ & $(0,4,0,0), \ (1,1,0,2), \ (1,0,3,0)$ \\
	\hline
	\rule[-1ex]{0pt}{2.5ex} Order $5$ & $(1,0,0,4), \ (0,2,3,0), \ (0,3,0,2)$ \\
	\hline
	\rule[-1ex]{0pt}{2.5ex} Order $6$ & $(0,1,3,2), \ (0,2,0,4), \ (0,0,6,0)$ \\
	\hline
	\rule[-1ex]{0pt}{2.5ex} Order $7$ & $(0,1,0,6), \ (0,0,3,4)$ \\
	\hline
	\rule[-1ex]{0pt}{2.5ex} Order $8$ & $(0,0,0,8)$ \\
	\hline
	\end{tabular}
	\caption{Terms expressed as $4$-tuples in the development \eqref{equation:order_8:2}}
	\label{figure:terms_exponent_1}
\end{table}

\begin{proof}
As before, we define the subspaces $F_0 := \mathbb{R}^d$ and by induction:
$$ F_k = \left\lbrace h \in F_{k-1} : \ \forall h' \in F_{k-1}, \  T_{2k} \cdot h \otimes h'^{\otimes 3} = 0 \right\rbrace $$
for $k=1,2,3$. We define $E_k$ as the orthogonal complement of $F_k$ in $F_{k-1}$ for $k=1,2,3$, so that
$$  \mathbb{R}^d = E_1 \oplus E_2 \oplus E_3 \oplus F_3 .$$

\begin{table}
  \centering
  \begin{tabular}{|c|c|c|c|}
    \hline
	$E_1$ & \multicolumn{3}{c|}{$F_1$} \\
	\hline
	\multirow{5}{*}{$T_2 \ge 0$} & \multicolumn{3}{c|}{$T_2=0$} \\
	\hhline{~---}
	& $E_2$ & \multicolumn{2}{c|}{$F_2$} \\
	\hhline{~---}
	& \multirow{3}{*}{$T_4 \ge 0$} & \multicolumn{2}{c|}{$T_4=0$} \\
	\hhline{~~--}
	& & $E_3$ & $F_3$ \\
	\hhline{~~--}
	& & $T_6 \ge 0$ & $T_6=0$ \\
	\hline 
  \end{tabular}
  \caption{Illustration of the subspaces}
  \label{figure:subspaces}
\end{table}

\noindent Then we apply a Taylor expansion up to order $8$ to the left side of \eqref{equation:order_8:2} and the multinomial formula \eqref{equation:multinomial}, which reads
$$ \sum_{k=2}^{8} \frac{1}{k!} \sum_{\substack{i_1,\ldots,i_{4} \in \lbrace 0,\ldots,k \rbrace \\ i_1 + \cdots + i_{4} = k }} \binom{k}{i_1,\ldots,i_{4}} t^{\frac{i_1}{2} + \cdots + \frac{i_4}{8} -1} T_k \cdot p_{_{E_1}}(h)^{\otimes i_1} \otimes p_{_{E_2}}(h)^{\otimes i_2} \otimes p_{_{E_3}}(h)^{\otimes i_3} \otimes p_{_{F_3}}(h)^{\otimes i_4} + o(1) .$$
$\triangleright$ If $\frac{i_1}{2} + \cdots + \frac{i_4}{8} > 1$ then the corresponding term is in $o(1)$ when $t \rightarrow 0$.

\noindent $\triangleright$ If $\frac{i_1}{2} + \cdots + \frac{i_4}{8} < 1$ then the corresponding term diverges when $t \rightarrow 0$, so we need to prove that actually
\begin{equation}
\label{eq:proof:order_8_null_terms}
T_k \cdot p_{_{E_1}}(h)^{\otimes i_1} \otimes p_{_{E_2}}(h)^{\otimes i_2} \otimes p_{_{E_3}}(h)^{\otimes i_3} \otimes p_{_{F_3}}(h)^{\otimes i_4} = 0 .
\end{equation}

-- If $\frac{i_1}{2} + \frac{i_2}{4} + \frac{i_3}{6} + \frac{i_4}{8} < 1$ but if we also have $\frac{i_1}{2} + \frac{i_2}{4} + \frac{i_3}{6} + \frac{i_4}{6} < 1$, then by applying the property at the order $6$ (Theorem \ref{theorem:order_6}) with the $3$-tuple $(i_1,i_2,i_3+i_4)$, we get \eqref{eq:proof:order_8_null_terms}.

-- So we only need to consider $4$-tuples such that $\frac{i_1}{2} + \frac{i_2}{4} + \frac{i_3}{6} + \frac{i_4}{8} < 1$ and $\frac{i_1}{2} + \frac{i_2}{4} + \frac{i_3}{6} + \frac{i_4}{6} \ge 1$. We can remove all the terms which are null by the definitions of the subspaces $E_1, \ E_2, \ E_3, \ F_3$. The remaining terms are:

For $k=4$: $\frac{t^{21/24}}{6} T_4 \cdot p_{_{E_1}}(h) \otimes p_{_{F_3}}(h)^{\otimes 3}$, $\frac{t^{11/12}}{2} T_4 \cdot p_{_{E_1}}(h) \otimes p_{_{E_3}}(h) \otimes p_{_{F_3}}(h)^{\otimes 2}$, $\frac{t^{23/24}}{2} T_4 \cdot p_{_{E_1}}(h) \otimes p_{_{E_3}}(h)^{\otimes 2} \otimes p_{_{F_3}}(h)$.

For $k=5$ : $\frac{t^{21/24}}{12} T_5 \cdot p_{_{E_2}}(h)^{\otimes 2} \otimes p_{_{F_3}}(h)^{\otimes 3}$, $\frac{t^{11/12}}{4} T_5 \cdot p_{_{E_2}}(h)^{\otimes 2} \otimes p_{_{E_3}}(h) \otimes p_{_{F_3}}(h)^{\otimes 2}$, $\frac{t^{23/24}}{4} T_5 \cdot p_{_{E_2}}(h)^{\otimes 2} \otimes p_{_{E_3}}(h)^{\otimes 2} \otimes p_{_{F_3}}(h)$.

For $k=6$ : $\frac{t^{21/24}}{5!} T_6 \cdot p_{_{E_2}}(h) \otimes p_{_{F_3}}(h)^{\otimes 5}$, $\frac{t^{11/12}}{4!} T_6 \cdot p_{_{E_2}}(h) \otimes p_{_{E_3}}(h) \otimes p_{_{F_3}}(h)^{\otimes 4}$, $\frac{t^{23/24}}{12} T_6 \cdot p_{_{E_2}}(h) \otimes p_{_{E_3}}(h)^{\otimes 2} \otimes p_{_{F_3}}(h)^{\otimes 3}$.

First, we note that
\begin{align*}
& \frac{1}{t^{21/24}} \left[ f(x^\star + t^{1/2}p_{_{E_1}}(h) + t^{1/4}p_{_{E_2}}(h) + t^{1/6}p_{_{E_3}}(h) + t^{1/8}p_{_{F_3}}(h)) - f(x^\star) \right] \\
\underset{t \rightarrow 0}{\longrightarrow} \ & \frac{1}{6} T_4 \cdot p_{_{E_1}}(h) \otimes p_{_{F_3}}(h)^{\otimes 3} + \frac{1}{12} T_5 \cdot p_{_{E_2}}(h)^{\otimes 2} \otimes p_{_{F_3}}(h)^{\otimes 3} + \frac{1}{5!} T_6 \cdot p_{_{E_2}}(h) \otimes p_{_{F_3}}(h)^{\otimes 5} \ge 0.
\end{align*}
Then, considering $h' = \lambda p_{_{E_1}}(h) + p_{_{E_2}}(h) + p_{_{E_3}}(h) + p_{_{F_3}}(h)$, we have that for all $\lambda \in \mathbb{R}$,
\begin{equation*}
\frac{\lambda}{6} T_4 \cdot p_{_{E_1}}(h) \otimes p_{_{F_3}}(h)^{\otimes 3} + \frac{1}{12} T_5 \cdot p_{_{E_2}}(h)^{\otimes 2} \otimes p_{_{F_3}}(h)^{\otimes 3} + \frac{1}{5!} T_6 \cdot p_{_{E_2}}(h) \otimes p_{_{F_3}}(h)^{\otimes 5} \ge 0 ,
\end{equation*}
so necessarily $$T_4 \cdot p_{_{E_1}}(h) \otimes p_{_{F_3}}(h)^{\otimes 3} = 0.$$
Then, considering $h' = p_{_{E_2}}(h) + \lambda p_{_{F_3}}(h)$ for $\lambda \in \mathbb{R}$, we get successively that the two other terms are null.

\medskip

Likewise, we prove successively that the terms in $t^{11/12}$ are null, and then that the terms in $t^{23/24}$ are null. This yields the convergence stated in \eqref{equation:order_8:2}.
\end{proof}

\subsection{Counter-example and proof of Theorem \ref{theorem:main} with $p\ge 5$ under the hypothesis \eqref{equation:even_terms_null:2}}
\label{section:order_10}

Algorithm \ref{algo:algorithm} may fail to yield such expansion of $f$ for orders no lower than $10$ if the hypothesis \eqref{equation:even_terms_null:2} is not fulfilled. Indeed for $p \ge 5$, there exist $p$-tuples $(i_1,\ldots,i_p)$ such that $\frac{i_1}{2}+\cdots+\frac{i_p}{2p} < 1$ and $i_1$, $\ldots$, $i_p$ are all even. Such tuples do not appear at orders $8$ and lower, but they do appear at orders $10$ and higher, for example $(0,2,0,0,4)$ for $k=6$.
In such a case, we cannot use the positiveness argument to prove that the corresponding term $T_k \cdot p_{_{E_1}}(h)^{\otimes i_1} \otimes \cdots \otimes p_{_{E_p}}(h)^{\otimes i_p}$ is zero, and in fact, it may be not zero.

Let us give a counter example. Consider
$$ \begin{array}{rrl}
f : & \mathbb{R}^2 & \longrightarrow \mathbb{R} \\
& (x,y) & \longmapsto x^4 + y^{10} + x^2 y^4.
\end{array} $$
Then $f \in \mathscr{A}_5^\star(0)$ and we have $E_1 = \lbrace 0 \rbrace$, $E_2 = \mathbb{R}\cdot(1,0)$, $E_3 = \lbrace 0 \rbrace$, $E_4 = \lbrace 0 \rbrace$, $F_4 = \mathbb{R}\cdot(0,1)$. But
$$ \frac{1}{t} f(t^{1/4}, t^{1/10}) = \frac{1}{t} \left(t + t + t^{9/10} \right) $$
goes to $+\infty$ when $t \rightarrow 0$.

\medskip

\textbf{Now, let us give the proof of Theorem \ref{theorem:main} for $p \ge 5$.} In this proof, we assume that the subspaces $E_1, \ \ldots, \ E_p$ given in Algorithm \ref{algo:algorithm} satisfy the hypothesis \eqref{equation:even_terms_null:2}.

\begin{proof}
We develop \eqref{equation:order_p:1}, which reads:
$$ \sum_{k=2}^{2p} \frac{1}{k!} \sum_{\substack{i_1,\ldots,i_{p} \in \lbrace 0,\ldots,k \rbrace \\ i_1 + \cdots + i_{p} = k }} \binom{k}{i_1,\ldots,i_{p}} t^{\frac{i_1}{2} + \cdots + \frac{i_p}{2p} -1} T_k \cdot p_{_{E_1}}(h)^{\otimes i_1} \otimes \cdots \otimes p_{_{E_p}}(h)^{\otimes i_p} + o(1) =: S .$$
The terms such that $\frac{i_1}{2} + \cdots + \frac{i_p}{2p} < 1$ may diverge when $t \rightarrow 0$, so let us prove that they are in fact null.
Let
$$ \alpha := \inf \left\lbrace \frac{i_1}{2}+\cdots+\frac{i_p}{2p} : \ h \longmapsto\sum_{k=2}^{2p} \frac{1}{k!} \sum_{\substack{i_1,\ldots,i_p \in \lbrace 0,\ldots,k \rbrace \\ i_1 + \cdots + i_{p} = k \\ \frac{i_1}{2} + \cdots + \frac{i_p}{2p}= \alpha }} \binom{k}{i_1,\ldots,i_{p}} T_k \cdot p_{_{E_1}}(h)^{\otimes i_1} \otimes \cdots \otimes p_{_{E_p}}(h)^{\otimes i_p} \not\equiv 0 \right\rbrace ,$$
and assume by contradiction that $\alpha < 1$. Then we have for all $h \in \mathbb{R}^d$:
$$ t^{1-\alpha} S \underset{t \rightarrow 0}{\longrightarrow} \left( \sum_{k=2}^{2p} \frac{1}{k!} \sum_{\substack{i_1,\ldots,i_p \in \lbrace 0,\ldots,k \rbrace \\ i_1 + \cdots + i_{p} = k \\ \frac{i_1}{2} + \cdots + \frac{i_p}{2p}= \alpha }} \binom{k}{i_1,\ldots,i_{p}} T_k \cdot p_{_{E_1}}(h)^{\otimes i_1} \otimes \cdots \otimes p_{_{E_p}}(h)^{\otimes i_p} \right) \ge 0,$$
by the local minimum property. Then, considering $h' = \lambda_1 p_{_{E_1}}(h) + \cdots + \lambda_p p_{_{E_p}}(h)$, we have, for all $h \in \mathbb{R}^d$ and $\lambda_1, \ \ldots, \ \lambda_d \in \mathbb{R}$,
\begin{equation}
\label{eq:order_5_polynomial}
\sum_{k=2}^{2p} \frac{1}{k!} \sum_{\substack{i_1,\ldots,i_p \in \lbrace 0,\ldots,k \rbrace \\ i_1 + \cdots + i_{p} = k \\ \frac{i_1}{2} + \cdots + \frac{i_p}{2p}= \alpha }} \lambda_1^{i_1}\ldots\lambda_p^{i_p} \binom{k}{i_1,\ldots,i_{p}} T_k \cdot p_{_{E_1}}(h)^{\otimes i_1} \otimes \cdots \otimes p_{_{E_p}}(h)^{\otimes i_p} \ge 0 .
\end{equation}

Now, we fix $h \in \mathbb{R}^d$ such that the polynomial in \eqref{eq:order_5_polynomial} in the variables $\lambda_1, \ \ldots, \ \lambda_p$ is not identically zero, and we consider $k_{\max}$ its highest homogeneous degree, so that we have
$$ \sum_{\substack{i_1,\ldots,i_p \in \lbrace 0,\ldots,k_{\max} \rbrace \\ i_1 + \cdots + i_{p} = k_{\max} \\ \frac{i_1}{2} + \cdots + \frac{i_p}{2p}= \alpha }} \lambda_1^{i_1}\ldots\lambda_p^{i_p} \binom{k_{\max}}{i_1,\ldots,i_{p}} T_{k_{\max}} \cdot p_{_{E_1}}(h)^{\otimes i_1} \otimes \cdots \otimes p_{_{E_p}}(h)^{\otimes i_p} \ge 0 .$$
If $k_{\max}$ is odd, this yields a contradiction, taking $\lambda_1 = \cdots = \lambda_p =: \lambda \rightarrow \pm \infty$. If $k_{\max}$ is even, we consider the index $l_1$ such that $i_{l_1} =: a_1$ is maximal and the coefficients in the above sum with $i_{l_1} = a_1$ are not all zero.
Then fixing all the $\lambda_l$ for $l \ne l_1$ and taking $\lambda_{l_1} \rightarrow \infty$, we have
$$ \sum_{\substack{i_1,\ldots,i_p \in \lbrace 0,\ldots,k_{\max} \rbrace \\ i_1 + \cdots + i_{p} = k_{\max} \\ \frac{i_1}{2} + \cdots + \frac{i_p}{2p}= \alpha \\ i_{l_1} = a_1 }} \lambda_1^{i_1}\ldots\lambda_p^{i_p} \binom{k_{\max}}{i_1,\cdots,i_{p}} T_{k_{\max}} \cdot p_{_{E_1}}(h)^{\otimes i_1} \otimes \cdots \otimes p_{_{E_p}}(h)^{\otimes i_p} \ge 0 .$$
Thus, if $a_1$ is odd, this yields a contradiction. If $a_1$ is even, we carry on this process by induction : knowing $l_1, \ \ldots, \ l_r$, we choose the index $l_{r+1}$ such that $l_{r+1} \notin \lbrace l_1,\ldots,l_r \rbrace$, the corresponding term
$$ \sum_{\substack{i_1,\ldots,i_p \in \lbrace 0,\ldots,k_{\max} \rbrace \\ i_1 + \cdots + i_{p} = k_{\max} \\ \frac{i_1}{2} + \cdots + \frac{i_p}{2p}= \alpha \\ i_{l_1}=a_1,\ldots,i_{l_{r+1}}=a_{r+1} }} \lambda_1^{i_1}\ldots\lambda_p^{i_p} \binom{k_{\max}}{i_1,\ldots,i_{p}} T_{k_{\max}} \cdot p_{_{E_1}}(h)^{\otimes i_1} \otimes \cdots \otimes p_{_{E_p}}(h)^{\otimes i_p} $$
is not identically null and such that $i_{l_{r+1}} =: a_{r+1}$ is maximal. Necessarily, $a_{r+1}$ is even.
In the end we will find a non-zero term whose exponents $i_{\ell}$ are all even which contradicts assumption \eqref{equation:even_terms_null:2}.
\end{proof}

\subsection{Proofs of the uniform convergence and of the non-constant property}
\label{subsec:unif_non_constant}

In this section we prove the additional properties claimed in Theorem \ref{theorem:main} : the uniform convergence with respect to $h$ on every compact set and the fact that the function $g$ is not constant in any of its variables $h_1, \ \ldots, \ h_d$.

\begin{proof}
First, let us prove that the convergence is uniform with respect to $h$ on every compact set. Let $\varepsilon >0$ and let $R>0$. By the Taylor formula at order $2p$, there exists $\delta>0$ such that for $||h|| < \delta$,
$$ \left| f(x^\star + h) - f(x^\star) - \sum_{k=2}^{2p} \frac{1}{k!} \sum_{i_1+\cdots+i_p=k} \binom{k}{i_1,\ldots,i_p} T_k \cdot p_{_{E_1}}(h)^{\otimes i_1} \otimes \cdots \otimes p_{_{E_p}}(h)^{\otimes i_p} \right| \le \varepsilon ||h||^{2p} .$$
Now, let us consider $t \rightarrow 0$ and $h \in \mathbb{R}^d$ with $||h|| \le R$. Then we have:
$$ \forall t \le \max\left(1, \left(\frac{\delta}{R}\right)^{1/(2p)}\right), \ ||t^{1/2}p_{_{E_1}}(h) + \cdots + t^{1/(2p)}p_{_{E_p}}(h) || \le \delta ,$$
so that
\begin{align*}
\left|\frac{1}{t} \left[ \right.\right. & \left. \left. f(x^\star + t^{1/2}p_{_{E_1}}(h) + \cdots + t^{1/(2p)}p_{_{E_p}}(h)) - f(x^\star)\right] - \sum_{k=2}^{2p} \frac{1}{k!} \sum_{i_1+\cdots+i_p=k} \binom{k}{i_1,\ldots,i_p} \right. \\
& \left. \cdot t^{\frac{i_1}{2}+\cdots+\frac{i_p}{2p}-1} T_k \cdot p_{_{E_1}}(h)^{\otimes i_1} \otimes \cdots \otimes p_{_{E_p}}(h)^{\otimes i_p} \right| \le \frac{\varepsilon}{t}||t^{1/2}p_{_{E_1}}(h) + \cdots + t^{1/(2p)}p_{_{E_p}}(h)||^{2p}.
\end{align*}
We proved or assumed that the terms such that $\frac{i_1}{2}+\cdots+\frac{i_p}{2p} < 1$ are zero. We denote by $g_1(h)$ the sum in the last equation with the terms such that $\frac{i_1}{2}+\cdots+\frac{i_p}{2p}=1$ and by $g_2(h)$ the sum with the terms such that $\frac{i_1}{2}+\cdots+\frac{i_p}{2p} > 1$. We also define $a$ as the smallest exponent of $t$ appearing in $g_2(h)$:
$$ a := \min \left\lbrace \frac{i_1}{2}+\cdots+\frac{i_p}{2p} \ : \ i_1,\ldots,i_p \in \lbrace 0,\ldots,2p\rbrace, \ i_1+\cdots+i_p \le 2p,  \ \frac{i_1}{2}+\cdots+\frac{i_p}{2p} > 1 \right\rbrace > 1 .$$
So that:
\begin{align}
\label{eq:uniform_proof:1}
& \left| \frac{1}{t} \left[ f(x^\star + t^{1/2}p_{E_1}(h) +\cdots + t^{1/(2p)}p_{_{E_p}}(h)) - f(x^\star) \right] - g_1(h) \right| \\
& \le t^{a-1}|g_2(h)| + \frac{\varepsilon}{t}||t^{1/2}p_{_{E_1}}(h) + \cdots + t^{1/(2p)}p_{_{E_p}}(h)||^{2p}. \nonumber
\end{align}
We remark that $h \mapsto g_2(h)$ is a polynomial function so is bounded on every compact set. We also have:
$$ \frac{\varepsilon}{t}||t^{1/2}p_{_{E_1}}(h) + \cdots + t^{1/(2p)}p_{_{E_p}}(h)||^{2p} \le \frac{\varepsilon (t^{1/(2p)})^{2p}}{t} ||h||^{2p} = \varepsilon ||h||^{2p} .$$
So \eqref{eq:uniform_proof:1} converges to $0$ as $t \rightarrow 0$, uniformly with respect to $h$ on every compact set.

\bigskip

Now let us assume that $f \in \mathscr{A}_p^\star(x^\star)$ ; we prove that the function $g$ defined in \eqref{eq:def_g} is not constant in any of its variables in the sense of \eqref{eq:def_non_constant}. Let $B \in \mathcal{O}_d(\mathbb{R})$ adapted to the decomposition $\mathbb{R}^d = E_1 \oplus \cdots \oplus E_p$. We have:
$$ \frac{1}{t} \left[ f\left( x^\star + B \cdot (t^\alpha \ast h)\right) - f(x^\star) \right] \underset{t \rightarrow 0}{\longrightarrow} g(h) .$$
Let $i \in \lbrace 1, \ldots, p \rbrace$ and $k$ such that $v_i := B \cdot e_i \in E_k$. Let us assume by contradiction that $g$ does not depend on the $i^{\text{th}}$ coordinate. Considering the expression of $g$ in \eqref{eq:def_g} and setting all the variables outside $E_k$ to $0$, we have:
$$ \forall h \in E_k, \ \lambda \in \mathbb{R} \mapsto T_{2k} \cdot (h + \lambda v_i)^{\otimes 2k} $$
is constant. Then applying \eqref{equation:multinomial}, we have:
$$ \forall h \in E_k, \ T_{2k} \cdot v_i \otimes h^{\otimes 2k-1} = 0 .$$
Moreover, for $h \in F_{k-1}$, let us write $h = h' + h''$ where $h' \in E_k$ and $h'' \in F_k$, so that
$$T_{2k} \cdot v_i \otimes h^{\otimes 2k-1} = T_{2k} \cdot v_i \otimes h'^{\otimes 2k-1} = 0,$$
where we used that $$ \forall h^{(3)} \in F_{k-1}, \ T_{2k} \cdot h'' \otimes \left(h^{(3)}\right)^{\otimes 2k-1} = 0$$
following \eqref{eq:F_k_def}, and Proposition \ref{proposition:null_tensor:1}. 
Considering the definition of $E_k$ as the orthogonal complement of $F_k$, which is defined in \eqref{eq:F_k_def}, the last equation contradicts that $v_i \in E_k$.
\end{proof}

\subsection{Non coercive case}
\label{section:non_coercive}

The function $g$ we obtain in Algorithm \ref{algo:algorithm} is a non-negative polynomial function which is constant in none of its variables. However, this does not always guarantee that $e^{-g} \in L^1(\mathbb{R}^d)$, or even that $g$ is coercive. Indeed, $g$ can be null on an unbounded continuous polynomial curve, while the polynomial degree of the minimum $x^\star$ of $f$ is higher than the degree of $g$ in these variables. For example, let us consider
\begin{align}
	\label{eq:non_coercive_f}
	f \colon \ \mathbb{R}^2 &\to \mathbb{R}\\
	(x,y) &\mapsto (x-y^2)^2 + x^6. \nonumber
\end{align}
Then $f \in \mathscr{A}_3^\star(0)$ and using Algorithm \ref{algo:algorithm}, we get
$$ g(x,y) = (x-y^2)^2 ,$$
which does not satisfy $e^{-g} \in L^1(\mathbb{R}^d)$.
In fact this case is highly degenerate, as, with
\begin{equation*}
f_\varepsilon(x,y) := f(x,y) + \varepsilon xy^2 = x^2 + y^4 - (2-\varepsilon)xy^2 + x^6 ,
\end{equation*}
we have that $g_\varepsilon(x,y) = x^2 + y^4 - (2-\varepsilon)xy^2$ satisfies $e^{-g_\varepsilon} \in L^1(\mathbb{R}^d)$ for every $\varepsilon \in (0,4)$ and that $x^\star$ is not the global minimum of $f_\varepsilon$ for every $\varepsilon \in (-\infty, 0) \cup (4, \infty)$.

We now prove that instead of assuming $e^{-g} \in L^1(\mathbb{R}^d)$, we can only assume that $g$ is coercive, which is justified in the following proposition. More specific conditions for $g$ to be coercive can be found in \cite{bajbar2015} and \cite{bajbar2019}.

\begin{proposition}
\label{prop:coercive}
Let $g : \mathbb{R}^d \rightarrow \mathbb{R}$ be the polynomial function obtained from Algorithm \ref{algo:algorithm}. If $g$ is coercive, then $e^{-g} \in L^1(\mathbb{R}^d)$.
\end{proposition}
\begin{proof}
Let
$$ A_k := \text{Span}\left(e_i: \ i \in \lbrace \dim(E_1) + \cdots + \dim(E_{k-1}) +1, \ldots, \dim(E_1)+\cdots+\dim(E_k) \rbrace \right) $$
for $k \in \lbrace 1,\ldots, p \rbrace$.
By construction of $g$, note that for all $t \in [0,+\infty)$,
$$ g\left(\sum_{k=1}^p t^{1/2k}p_{_{A_k}}(h)\right) = tg(h) .$$
Since $g$ is coercive, there exists $R \ge 1$ such that for every $h$ with $||h|| \ge R$, $g(h) \ge 1$.
Then, for every $h \in \mathbb{R}^d$, we have:
\begin{align*}
g(h) & = g\left( \sum_{k=1}^p p_{_{A_k}}(h) \right) = g\left( \sum_{k=1}^p \frac{||h||^{1/2k}}{R^{1/2k}}p_{_{A_k}}\left(R^{1/2k}\frac{h}{||h||^{1/2k}}\right) \right) \\
& = \frac{||h||}{R} g\left( \sum_{k=1}^p p_{_{A_k}}\left(R^{1/2k}\frac{h}{||h||^{1/2k}}\right) \right).
\end{align*}
Then, for $||h|| \ge R$,
\begin{align*}
\left|\left|  \sum_{k=1}^p p_{_{A_k}}\left(R^{1/2k}\frac{h}{||h||^{1/2k}}\right) \right| \right|^2 = \sum_{k=1}^p \frac{R^{1/k}}{||h||^{1/k}} ||p_{_{A_k}}(h)||^2 \ge \frac{R}{||h||} ||h||^2 = R ||h|| \ge R^2 \ge R,
\end{align*}
so that $g(h) \ge \frac{||h||}{R}$ which in turn implies $e^{-g} \in L^1(\mathbb{R}^d)$.
\end{proof}

We now deal with the simplest configuration where the function $g$ is not coercive, as described in \eqref{eq:non_coercive_sum}, by dealing with the case where $f$ is given by \eqref{eq:non_coercive_f}, which is an archetype of such configuration. However, dealing with the general case is more complicated and to give a general formula for the rate of convergence of the measure $\pi_t$ in this case is not our current objective.

\begin{proposition}
\label{prop:x_y2_convergence}
Let the function $f$ be given by \eqref{eq:non_coercive_f}. Then, if $(X_t,Y_t) \sim C_t e^{-f(x,y)/t} dx dy$, we have:
\begin{equation*}
\left( \frac{X_t}{t^{1/6}}, \frac{Y_t^2-X_t}{t^{1/2}} \right) \underset{t \rightarrow 0}{\longrightarrow} C\frac{e^{-x^6}}{\sqrt{x}} \frac{e^{-y^2}}{\sqrt{\pi}} \mathds{1}_{x \ge 0} dx dy ,
\end{equation*}
where $C = \left( \int_0^\infty \frac{e^{-x^6}}{\sqrt{x}}dx \right)^{-1}$.
\end{proposition}
\begin{proof}
First, let us consider the normalizing constant $C_t$. We have :
\begin{align*}
C_t^{-1} & = \int_{\mathbb{R}^2} e^{-\frac{(x-y^2)^2+x^6}{t}} dx dy = 2t^{3/4} \int_{-\infty}^{\infty} e^{-t^2 x^6} \int_0^\infty e^{-(y^2-x)^2}dy \ dx \\
& = t^{3/4} \int_{-\infty}^{\infty} e^{-t^2 x^6} \int_{-x}^\infty \frac{e^{-u^2}}{\sqrt{u+x}}dy \ dx
= t^{7/12} \int_{-\infty}^{\infty} e^{-x^6} \int_{-t^{-1/3}x}^{\infty} \frac{e^{-u^2}}{\sqrt{t^{1/3}u+x}} du \ dx \\
& \underset{t \rightarrow 0}{\sim} t^{7/12} \int_{0}^{\infty} \frac{e^{-x^6}}{\sqrt{x}} \int_{-\infty}^{\infty} e^{-u^2} du \ dx,
\end{align*}
where the convergence is obtained by dominated convergence and where we performed the change of variables $x' = t^{-1/6}x$ and $u = t^{-1/2}(y^2-x)$. Then we consider, for $a_1 < b_1$ and $a_2 < b_2$,
$$ \mathbb{P}\left( \frac{X_t}{t^{1/6}} \in [a_1,b_1], \ \frac{Y^2-X}{t^{1/2}} \in [a_2,b_2] \right) .$$
Performing the same changes of variables and using the above equivalent of $C_t$ completes the proof.
\end{proof}

More generally, if the function $g$ is not coercive and if we can write, up to a change of basis,
\begin{equation}
\label{eq:non_coercive_sum}
g(h_1,\ldots,h_d) = Q_1(h_1,h_2)^2 + Q_2(h_3,h_4)^2 + \cdots + Q_r(h_{2r-1},h_{2r})^2 + \widetilde{g}(h_{2r+1},\ldots,h_d) ,
\end{equation}
where the $Q_i$ are polynomials with two variables null on an unbounded curve (for example, $Q_i(x,y) = (x-y^2)$, $Q_i(x,y) = (x^2-y^3)$, $Q_i(x,y)=x^2y^2$), and where $\widetilde{g}$ is a non-negative coercive polynomial, then
\begin{align*}
& \left( a_1\left((X_t)_{1}, (X_t)_{2}, t \right), \ldots, a_r\left((X_t)_{2r-1}, (X_t)_{2r}, t \right), \left(\frac{1}{t^{\alpha_{2r+1}}},\ldots,\frac{1}{t^{\alpha_d}}\right) \ast \left(\widetilde{B} \cdot ((X_t)_{2r+1},\ldots,(X_t)_d)\right) \right) \\
& \underset{t \rightarrow 0}{\longrightarrow} b_1(x_1,x_2) \ldots b_r(x_{2r-1},x_{2r}) Ce^{-\widetilde{g}(x_{2r+1},\ldots,x_d)} dx_1 \ldots dx_{2r} dx_{2r+1}\cdots dx_d ,
\end{align*}
where $C$ is a normalization constant, $\widetilde{B} \in \mathcal{O}_{d-2r-1}(\mathbb{R})$ is an orthogonal transformation and for all $k=1,\ldots,r$, $a_k : \mathbb{R}^2 \times (0,+\infty) \rightarrow \mathbb{R}^2$ and $b_k$ is a density on $\mathbb{R}^2$. Such $a_k$ and $b_k$ can be obtained by applying the same method as in Proposition \ref{prop:x_y2_convergence}.
Algorithm \ref{algo:algorithm} yields the first change of variable for this method, given by the exponents $(\alpha_i)$ (in the proof of Proposition \ref{prop:x_y2_convergence}, the first change of variable is $t^{-1/2}x$ and $t^{-1/4}y$) and thus seems to be the first step of a more general procedure in this case.
However, we do not give a general formula as the general case is cumbersome. Moreover, we do not give a method where the non coercive polynomials $Q_i$ depend on more than two variables, like
$$ Q(x,y,z) = (x-y^2)^2 + (x-z^2)^2 .$$
The method sketched in Proposition \ref{prop:x_y2_convergence} cannot be direclty applied to this case.

\section{Proofs of Theorem \ref{theorem:single_well} and Theorem \ref{theorem:multiple_well} using Theorem \ref{theorem:main}}
\label{section:proofs_athreya}

\subsection{Single well case}
\label{subsec:single_well_proof}

We now prove Theorem \ref{theorem:single_well}.

\begin{proof}
Using Theorem \ref{theorem:main}, we have for all $h \in \mathbb{R}^d$:
$$ \frac{1}{t} f(B \cdot (t^\alpha \ast h)) \underset{t \rightarrow 0}{\longrightarrow} g(h) .$$
To simplify the notations, assume that there is no need of a change of basis i.e. $B=I_d$. We want to apply Theorem \ref{theorem:athreya:1} to the function $f$. However the condition 
$$ \int_{\mathbb{R}^d} \sup_{0<t<1} e^{-\frac{f \left(t^{\alpha_1}h_1,\ldots,t^{\alpha_d}h_d\right)}{t}} dh_1\ldots dh_d < \infty $$
is not necessarily true. Instead, let $\varepsilon > 0$ and we apply Theorem \ref{theorem:athreya:1} to $\widetilde{f}$, where $\widetilde{f}$ is defined as:
$$ \widetilde{f}(h) = \left\lbrace \begin{array}{ll}
f(h) & \text{ if } h \in \mathcal{B}(0,\delta) \\
||h||^2 & \text{ else} ,
\end{array} \right.  $$
and where $\delta > 0$ will be fixed later. Then $\widetilde{f}$ satisfies the hypotheses of Theorem \ref{theorem:athreya:1}. The only difficult point to prove is the last condition of Theorem \ref{theorem:athreya:1}. 
If $t \in (0,1]$ and $h \in \mathbb{R}^d$ are such that $(t^{\alpha_1}h_1,\ldots,t^{\alpha_d}h_d) \notin \mathcal{B}(0,\delta)$, then
$$ \frac{\widetilde{f} (t^{\alpha_1}h_1,\ldots,t^{\alpha_d}h_d)}{t} = \frac{||(t^{\alpha_1}h_1,\ldots,t^{\alpha_d}h_d) ||^2}{t} \ge ||h||^2, $$
because for all $i$, $\alpha_i \le \frac{1}{2}$. If $t$ and $h$ are such that $(t^{\alpha_1}h_1,\ldots,t^{\alpha_d}h_d)  \in \mathcal{B}(0,\delta)$, then choosing $\delta$ such that for all $(t^{\alpha_1}h_1,\ldots,t^{\alpha_d}h_d) \in \mathcal{B}(0,\delta)$,
$$ \left| \frac{f (t^{\alpha_1}h_1,\ldots,t^{\alpha_d}h_d)}{t} - g(h) \right| \le \varepsilon, $$
which is possible because of the uniform convergence on every compact set (see Section \ref{subsec:unif_non_constant}), we derive that
$$ \frac{f (t^{\alpha_1}h_1,\ldots,t^{\alpha_d}h_d)}{t} \ge g(h) - \varepsilon .$$
Hence
$$ \int_{\mathbb{R}^d} \sup_{0<t<1} e^{-\frac{\widetilde{f} \left(t^{\alpha_1}h_1,\ldots,t^{\alpha_d}h_d\right)}{t}} dh_1\ldots dh_d \le \int_{\mathbb{R}^d} e^{-||h||^2} dh + e^{\varepsilon} \int_{\mathbb{R}^d} e^{-g(h)} dh .$$
Since $g$ is coercive, using Proposition \ref{prop:coercive} we have $e^{-g} \in L^1(\mathbb{R}^d)$ and it follows from Theorem \ref{theorem:athreya:1} that if $\widetilde{X}_t$ has density $\widetilde{\pi}_t(x) := \widetilde{C}_t e^{-\widetilde{f}(x)/t}$, then
$$ \left( \frac{(\widetilde{X}_t)_1}{t^{\alpha_1}}, \ldots, \frac{(\widetilde{X}_t)_d}{t^{\alpha_d}} \right) \overset{\mathscr{L}}{\longrightarrow} X  \ \text{ as } t \rightarrow 0 ,$$
where $X$ has density proportional to $e^{-g(x)}$.

Now, let us prove that if $X_t$ has density proportional to $e^{-f(x)/t}$, then we also have
\begin{equation}
\label{equation:Y_law_convergence}
\left( \frac{(X_t)_1}{t^{\alpha_1}}, \ldots, \frac{(X_t)_d}{t^{\alpha_d}} \right) \overset{\mathscr{L}}{\longrightarrow} X  \ \text{ as } t \rightarrow 0 .
\end{equation}
Let $\varphi : \mathbb{R}^d \rightarrow \mathbb{R}$ be continuous with compact support. Then
\begin{align*}
& \mathbb{E}\left[ \varphi\left( \frac{(X_t)_1}{t^{\alpha_1}}, \ldots, \frac{(X_t)_d}{t^{\alpha_d}} \right) - \varphi\left( \frac{(\widetilde{X}_t)_1}{t^{\alpha_1}}, \ldots, \frac{(\widetilde{X}_t)_d}{t^{\alpha_d}} \right) \right] \\
& = \int_{\mathbb{R}^d} \varphi \left(\frac{x_1}{t^{\alpha_1}},\ldots,\frac{x_d}{t^{\alpha_d}}\right) \left(C_t e^{-\frac{f(x_1,\ldots,x_d)}{t}} - \widetilde{C}_t e^{-\frac{\widetilde{f}(x_1,\ldots,x_d)}{t}} \right) dx_1\ldots dx_d =: I_1 + I_2,
\end{align*}
where $I_1$ is the integral on the set $\mathcal{B}(0,\delta)$ and $I_2$ on $\mathcal{B}(0,\delta)^c$. We have then:
$$ |I_2| \le || \varphi ||_\infty ( \pi_t(\mathcal{B}(0,\delta)^c) + \widetilde{\pi}_t(\mathcal{B}(0,\delta)^c) ) \underset{t \rightarrow 0}{\longrightarrow} 0 ,$$
where we used Proposition \ref{proposition:gibbs}. On the other hand, we have $f=\widetilde{f}$ on $\mathcal{B}(0,\delta)$, so that
$$ |I_1| \le ||\varphi||_\infty |C_t - \widetilde{C}_t| \int_{\mathcal{B}(0,\delta)} e^{-\frac{f(x)}{t}}dx \le ||\varphi||_\infty \left|1 - \frac{\widetilde{C}_t}{C_t}\right| .$$
And we have:
$$ \frac{\widetilde{C}_t}{C_t} = \frac{\int e^{-\frac{f(x)}{t}}dx}{\int e^{-\frac{\widetilde{f}(x)}{t}}dx} = \frac{\int_{\mathcal{B}(0,\delta)} e^{-\frac{f(x)}{t}}dx + \int_{\mathcal{B}(0,\delta)^c} e^{-\frac{f(x)}{t}}dx}{\int_{\mathcal{B}(0,\delta)} e^{-\frac{f(x)}{t}}dx + \int_{\mathcal{B}(0,\delta)^c} e^{-\frac{\widetilde{f}(x)}{t}}dx} .$$
By Proposition \ref{proposition:gibbs}, we have when $t \rightarrow 0$
\begin{align*}
\int_{\mathcal{B}(0,\delta)^c} e^{-\frac{\widetilde{f}(x)}{t}}dx & = o\left( \int_{\mathcal{B}(0,\delta)} e^{-\frac{\widetilde{f}(x)}{t}}dx \right) \\
\int_{\mathcal{B}(0,\delta)^c} e^{-\frac{f(x)}{t}}dx & = o\left( \int_{\mathcal{B}(0,\delta)} e^{-\frac{f(x)}{t}}dx \right),
\end{align*}
so that $\widetilde{C}_t/C_t \rightarrow 1$, so $I_1 \rightarrow 0$, which then implies \eqref{equation:Y_law_convergence}.
\end{proof}

\subsection{Multiple well case}

We now prove Theorem \ref{theorem:multiple_well}.

\begin{proof}
The first point is a direct application of Theorem \ref{theorem:athreya:2}.
For the second point, we remark that $X_{it}$ has a density proportional to $e^{-f_i(x)/t}$, where
$$ f_i(x) := \left\lbrace\begin{array}{l}
f(x) \text{ if } x \in \mathcal{B}(x_i, \delta) \\
+ \infty \text{ else}.
\end{array} \right. $$
We then consider $\widetilde{f}_i$ as in Section \ref{subsec:single_well_proof}:
$$ \widetilde{f}_i(x) = \left\lbrace \begin{array}{ll}
f_i(x) & \text{ if } x \in \mathcal{B}(x_i,\delta) \\
||x||^2 & \text{ else} .
\end{array} \right. $$
and still as in Section \ref{subsec:single_well_proof}, we apply Theorem \ref{theorem:athreya:1} to $\widetilde{f}_i$ and then prove that random variables with densities proportional to $e^{-\widetilde{f}_i(x)/t}$ and $e^{-f_i(x)/t}$ respectively have the same limit in law.
\end{proof}

\section{Infinitely flat minimum}
\label{section:flat}

In this section, we deal with an example of infinitely flat global minimum, where we cannot use a Taylor expansion.

\begin{proposition}
Let $f : \mathbb{R}^d \rightarrow \mathbb{R}$ such that
$$ \forall x \in \mathcal{B}(0,1), \ f(x) = e^{-\frac{1}{||x||^2}} $$
and
$$ \forall x \notin \mathcal{B}(0,1), \ f(x) > a $$
for some $a > 0$. Furthermore, assume that $f$ is coercive and $e^{-f} \in L^1(\mathbb{R}^d)$. Then, if $X_t$ has density $\pi_t$,
$$ \log^{1/2}\left(\frac{1}{t}\right) \cdot X_t \overset{\mathscr{L}}{\longrightarrow} X \ \text{ as } t \rightarrow 0 ,$$
where $X \sim \mathcal{U}(\mathcal{B}(0,1))$.
\end{proposition}
\begin{proof}
Noting that $\int_{||x|| >1}e^{-f(x)/t}dx \rightarrow 0$ as $t \rightarrow 0$ by dominated convergence, we have
$$ C_t \underset{t \rightarrow 0}{\sim} \left(\int_{\mathcal{B}(0,1)} e^{-e^{-\frac{1}{||x||^2}}/t} dx \right)^{-1} = \log^{d/2}\left(\frac{1}{t}\right) \left( \underbrace{\int_{\mathcal{B}(0,\sqrt{\log(1/t)})} e^{-t^{\frac{1}{||x||^2}-1}} dx}_{\underset{t \rightarrow 0}{\rightarrow} \text{Vol}(\mathcal{B}(0,1)) } \right)^{-1}, $$
where the convergence of the integral is obtained by dominated convergence. Then we have, for $-1<a_i<b_i<1$ and $\sum_i a_i^2 < 1$, $\sum_i b_i^2 < 1$:
\begin{align*}
\mathbb{P}\left(\log^{1/2}\left(\frac{1}{t}\right) \cdot X_t \in \prod_{i=1}^d [a_i,b_i] \right) = \frac{C_t}{\log^{d/2}\left(\frac{1}{t}\right)}\int_{(a_i)}^{(b_i)} e^{-t^{\frac{1}{|x|^2}-1}} dx \underset{t \rightarrow 0}{\longrightarrow} \frac{\prod_{i=1}^d (b_i-a_i)}{\text{Vol}(\mathcal{B}(0,1))}.
\end{align*}
\end{proof}

\section*{Acknowledgements}
\noindent I would like to thank Gilles Pag\`es for insightful discussions.

\newcommand{\etalchar}[1]{$^{#1}$}

\appendix

\section{Properties of tensors}

\begin{proposition}
\label{proposition:null_tensor:1}
Let $T_k$ be a symmetric tensor of order $k$ in $\mathbb{R}^d$. Let $E$ be a subspace of $\mathbb{R}^d$. Assume that
$$ \forall h \in E, \ T_k \cdot h^{\otimes k} = 0 .$$
Then we have
$$ \forall h_1,\ldots,h_k \in E, \ T_k \cdot h_1 \otimes \cdots \otimes h_k = 0 .$$
\end{proposition}

\begin{proof}
Using \eqref{equation:multinomial}, we have for $h_1, \ \ldots, \ h_k \in E$ and $\lambda_1, \ \ldots, \ \lambda_k \in \mathbb{R}$,
$$ T_k \cdot (\lambda_1 h_1+\cdots+\lambda_k h_k)^{\otimes k} = \sum_{i_1+\cdots+i_k=k} \binom{k}{i_1,\ldots,i_k} \lambda_1^{i_1}\ldots\lambda_k^{i_k} T_k \cdot h_1^{\otimes i_1} \otimes \cdots \otimes h_k^{\otimes i_k} = 0 ,$$
which is an identically null polynomial in the variables $\lambda_1, \ \ldots, \ \lambda_k$, so every coefficient is null, in particular
$$ \forall h_1,\ldots,h_k \in E, \ T_k \cdot h_1 \otimes \cdots \otimes h_k = 0 .$$
\end{proof}


\begin{thebibliography}{CGLM08}

\bibitem[AH10]{athreya2010}
Krishna~B. Athreya and Chii-Ruey Hwang.
\newblock Gibbs measures asymptotics.
\newblock {\em Sankhya A}, 72(1):191--207, February 2010.

\bibitem[{Bar}20]{barrera2020}
Gerardo {Barrera}.
\newblock {Limit behavior of the invariant measure for Langevin dynamics}.
\newblock {\em arXiv e-prints}, page arXiv:2006.06808, 2020.

\bibitem[BS15]{bajbar2015}
Tomas Bajbar and Oliver Stein.
\newblock {Coercive Polynomials and Their {N}ewton Polytopes}.
\newblock {\em SIAM Journal on Optimization}, 25:1542--1570, January 2015.

\bibitem[BS19]{bajbar2019}
Tomas Bajbar and Oliver Stein.
\newblock Coercive polynomials: stability, order of growth, and {N}ewton
  polytopes.
\newblock {\em Optimization}, 68(1):99--124, 2019.

\bibitem[CGLM08]{comon2008}
Pierre Comon, Gene Golub, Lek-Heng Lim, and Bernard Mourrain.
\newblock Symmetric tensor and symmetric tensor rank.
\newblock {\em SIAM Journal on matrix analysis and applications}, 2008.

\bibitem[{Dal}14]{dalalyan2016}
Arnak~S. {Dalalyan}.
\newblock {Theoretical guarantees for approximate sampling from smooth and
  log-concave densities}.
\newblock {\em arXiv e-prints}, page arXiv:1412.7392, 2014.

\bibitem[DPG{\etalchar{+}}14]{dauphin2014}
Yann {Dauphin}, Razvan {Pascanu}, Caglar {Gulcehre}, Kyunghyun {Cho}, Surya
  {Ganguli}, and Yoshua {Bengio}.
\newblock {Identifying and attacking the saddle point problem in
  high-dimensional non-convex optimization}.
\newblock {\em arXiv e-prints}, page arXiv:1406.2572, 2014.

\bibitem[FP99]{fort1999}
Jean-Claude Fort and Gilles Pagès.
\newblock {Asymptotic Behavior of a Markovian Stochastic Algorithm with
  Constant Step}.
\newblock {\em SIAM Journal on Control and Optimization}, 37(5):1456--1482,
  1999.

\bibitem[GM90]{gelfand-mitter}
Saul~B. {Gelfand} and Sanjoy~K. {Mitter}.
\newblock Recursive stochastic algorithms for global optimization in
  $\mathbb{R}^d$.
\newblock In {\em 29th IEEE Conference on Decision and Control}, pages 220--221
  vol.1, 1990.

\bibitem[Hil88]{hilbert1888}
David Hilbert.
\newblock {Ueber die Darstellung definiter Formen als Summe von
  Formenquadraten}.
\newblock {\em Mathematische Annalen}, 1888.

\bibitem[Hwa80]{hwang1980}
Chii-Ruey Hwang.
\newblock {L}aplace's method revisited: Weak convergence of probability
  measures.
\newblock {\em Ann. Probab.}, 8(6):1177--1182, December 1980.

\bibitem[Hwa81]{hwang1981}
Chii-Ruey Hwang.
\newblock {A Generalization of {L}aplaces's Method}.
\newblock {\em Proceedings of the American Mathematical Society},
  82(3):446--451, 1981.

\bibitem[Kha12]{khasminskii2012}
Rafail Khasminskii.
\newblock {\em {Stochastic stability of differential equations. With
  contributions by G. N. Milstein and M. B. Nevelson. 2nd completely revised
  and enlarged ed}}, volume~66.
\newblock Springer-Verlag Berlin Heidelberg, 2012.

\bibitem[LA87]{laarhoven1987}
Peter J.~M. Laarhoven and Emile H.~L. Aarts.
\newblock {\em {Simulated Annealing: Theory and Applications}}.
\newblock Kluwer Academic Publishers, USA, 1987.

\bibitem[Laz92]{lazarev1992}
Vladimir~A. Lazarev.
\newblock {Convergence of Stochastic-Approximation procedures in the case of a
  Regression Equation with Several Roots}.
\newblock {\em Probl. Peredachi Inf.}, 28(1):75--88, 1992.

\bibitem[LCCC15]{li2015}
Chunyuan {Li}, Changyou {Chen}, David {Carlson}, and Lawrence {Carin}.
\newblock {Preconditioned Stochastic Gradient Langevin Dynamics for Deep Neural
  Networks}.
\newblock {\em arXiv e-prints}, page arXiv:1512.07666, 2015.

\bibitem[Mot67]{motzkin1967}
Theodore~S. Motzkin.
\newblock The arithmetic-geometric inequality.
\newblock In {\em Inequalities ({P}roc. {S}ympos. {W}right-{P}atterson {A}ir
  {F}orce {B}ase, {O}hio, 1965)}, pages 205--224. Academic Press, New York,
  1967.

\bibitem[WT11]{welling2011}
Max Welling and Yee~Whye Teh.
\newblock {Bayesian Learning via Stochastic Gradient Langevin Dynamics}.
\newblock In {\em Proceedings of the 28th International Conference on
  International Conference on Machine Learning}, ICML'11, page 681–688.
  Omnipress, 2011.

\end{thebibliography}
\end{document}